\documentclass[12pt]{amsart}

\usepackage{amsmath}
\usepackage{amssymb}
\usepackage{verbatim}
\usepackage{graphicx}
\usepackage{amscd}
\usepackage[latin1]{inputenc}
\usepackage{amsfonts}
\usepackage{latexsym}
\usepackage{epsfig}
\newtheorem{theorem}{Theorem}[section]
\newtheorem{corollary}[theorem]{Corollary}
\newtheorem{lemma}[theorem]{Lemma}
\newtheorem{question}[theorem]{Question}

\theoremstyle{definition}
\newtheorem{definition}[theorem]{Definition}

\begin{document}
\title[TCPE and ZTCPE]{Topologically completely positive entropy and zero-dimensional topologically completely positive entropy}

\begin{abstract}
In a previous paper (\cite{pavlovTCPE}), the author gave a characterization for when a $\mathbb{Z}^d$-shift of finite type (SFT) 
has no nontrivial subshift factors with zero entropy, a property which we here call zero-dimensional topologically completely positive entropy 
(ZTCPE). In this work, we study the difference between this notion and the more classical topologically completely positive entropy (TCPE) of Blanchard. We show that there are one-dimensional subshifts and two-dimensional SFTs which have ZTCPE but not TCPE. In addition, we show that strengthening the hypotheses of the main result of \cite{pavlovTCPE} yields a sufficient condition for a $\mathbb{Z}^d$-SFT to have TCPE.
\end{abstract}

\date{}
\author{Ronnie Pavlov}
\address{Ronnie Pavlov\\
Department of Mathematics\\
University of Denver\\
2280 S. Vine St.\\
Denver, CO 80208}
\email{rpavlov@du.edu}
\urladdr{www.math.du.edu/$\sim$rpavlov/}
\thanks{The author acknowledges the support of NSF grant DMS-1500685.}
\renewcommand{\subjclassname}{MSC 2010}
\maketitle


\section{Introduction}\label{intro}

This work is motivated by an unfortunate misuse of the term ``topologically completely positive entropy'' (hereafter called TCPE) in some works written or co-written by the author (see \cite{BPS}, \cite{pavlovTCPE}). Blanchard (\cite{blanchard}) originally defined TCPE to mean that a topological dynamical system has no nontrivial (i.e. containing more than one point) factors with zero entropy. However, in \cite{BPS} and 
\cite{pavlovTCPE}, the proofs furnished were for $\mathbb{Z}^d$ subshifts and only proved that all nontrivial subshift factors have positive entropy. It is quite simple to show that any system has a nontrivial subshift factor with zero entropy if and only if it has a nontrivial zero-dimensional factor with zero entropy (see Theorem~\ref{STCPE=ZTCPE}), and so we say a system has zero-dimensional TCPE (or ZTCPE) if all nontrivial zero-dimensional factors have positive entropy. It is obvious that TCPE implies ZTCPE, and just as obvious that the converse is false in general: any topological dynamical system on a connected space trivially has ZTCPE since it has no nontrivial zero-dimensional factors. 

However, it is natural to wonder whether or not the two notions coincide if $(X,T)$ is itself assumed to be zero-dimensional, a subshift, or even a shift of finite type. This is not the case; we will construct several examples of systems in these classes with ZTCPE but not TCPE. We prove the following results in the one-dimensional case.

\begin{theorem}\label{1dsub}
There exists a $\mathbb{Z}$-subshift which has ZTCPE but not TCPE.
\end{theorem}

\begin{theorem}\label{1dSFT}
Any $\mathbb{Z}$-SFT with ZTCPE also has TCPE.
\end{theorem}

In the two-dimensional case, the picture is even more interesting. The following theorem was proved by the author in \cite{pavlovTCPE}; it erroneously purported to give conditions equivalent to TCPE for multidimensional SFTs
(and had an unfortunate typo which replaced ``SFT'' with ``subshift''), but we've given the corrected version below.

\begin{theorem}{\rm (\cite{pavlovTCPE}, Theorem 1.1)}\label{ZTCPEthm}
A $\mathbb{Z}^d$-SFT has ZTCPE if and only if it has the following two properties: every $w \in L(X)$ has positive measure for some $\mu \in \mathcal{M}(X)$, and for every $S \subset \mathbb{Z}^d$ and $w,w' \in L_S(X)$, there exist patterns $w = w_1, w_2, \ldots, w_n = w'$ so that for $1 \leq i < n$, there exist homoclinic points $x,x' \in X$ with $x(S) = w_i$ and $x'(S) = w_{i+1}$.
\end{theorem}

The second property in Theorem~\ref{ZTCPEthm} was called ``chain exchangeability'' of $w$ and $w'$ in \cite{pavlovTCPE}. Note that even if all pairs of patterns with the same shape are chain exchangeable, it is theoretically possible that the number $n$ of required ``exchanges'' could increase with the size of the patterns. In fact this is related to TCPE as well, as shown by the following theorem.

\begin{theorem}\label{TCPEthm}
If a $\mathbb{Z}^d$-SFT satisfies the hypotheses of Theorem~\ref{ZTCPEthm} with a uniform bound on the required $n$ over all patterns $w,w'$, then $X$ has TCPE.
\end{theorem}

Theorem~\ref{TCPEthm} implies that the previously mentioned results of \cite{BPS} and \cite{pavlovTCPE} in fact do yield TCPE for the subshifts in question (see Corollaries~\ref{bgTCPE} and \ref{ribbonTCPE}), since all of those proofs included such a uniform bound on $n$. The remaining question of whether ZTCPE in fact implies TCPE for $\mathbb{Z}^d$-SFTs is answered negatively by the following.

\begin{theorem}\label{2dSFT}
There exists a $\mathbb{Z}^2$-SFT $X$ which has ZTCPE but not TCPE.
\end{theorem}

We note that by necessity, the $X$ from Theorem~\ref{2dSFT} has the property that all pairs of patterns are chain exchangeable, but that larger and larger patterns may require more and more exchanges.
We do not know whether this property is sufficient as well as necessary, i.e. whether the converse of Theorem~\ref{TCPEthm} holds as well.

\begin{question}\label{TCPEQ}
Does every $\mathbb{Z}^d$ SFT with TCPE satisfy the hypotheses of Theorem~\ref{ZTCPEthm} with a uniform bound on $n$ over all $w,w'$?
\end{question}

\section*{Acknowledgments} 
The author would like to thank Benjy Weiss for pointing out the fact that only ZTCPE was proved in \cite{pavlovTCPE}, and Mike Boyle for many useful discussions about ZTCPE, which led to Theorem~\ref{STCPE=ZTCPE} and a preliminary version of the $\mathbb{Z}$-subshift example proving Theorem~\ref{1dsub}.

\section{Definitions}\label{defs}

We begin with some definitions from topological/symbolic dynamics.

\begin{definition}
A $\mathbb{Z}^d$ \textbf{topological dynamical system} $(X, T_v)$ is given by a compact metric space $X$ and a $\mathbb{Z}^d$ action 
$\{T_v\}_{v \in \mathbb{Z}^d}$ by homeomorphisms on $X$. In the special case $d = 1$, it is standard to refer to the system as $(X,T)$ rather than $(X, T_n)$; the single homeomorphism $T$ generates the entire action in this case anyway.
\end{definition}

\begin{definition}
For any finite set $A$ (called an \textbf{alphabet}), the \textbf{$\mathbb{Z}^d$-shift action} on $A^{\mathbb{Z}^d}$, denoted by $\{\sigma_t\}_{t \in \mathbb{Z}^d}$, is defined by $(\sigma_t x)(s) = x(s+t)$ for $s,t \in \mathbb{Z}^d$. 
\end{definition}

We always endow $A^{\mathbb{Z}^d}$ with the product discrete topology, with respect to which it is obviously compact metric. 

\begin{definition}
A \textbf{$\mathbb{Z}^d$-subshift} is a closed subset of $A^{\mathbb{Z}^d}$ which is invariant under the $\mathbb{Z}^d$-shift action. When dimension is clear from context, we often just use the term \textbf{subshift}.
\end{definition}

Any $\mathbb{Z}^d$-subshift inherits a topology from $A^{\mathbb{Z}^d}$, and is compact. Each $\sigma_t$ is a homeomorphism on any 
$\mathbb{Z}^d$-subshift, and so any $\mathbb{Z}^d$-subshift, when paired with the $\mathbb{Z}^d$-shift action, is a topological dynamical system. Where it will not cause confusion, we suppress the action $\sigma_v$ and just refer to a subshift by the space $X$. 

\begin{definition}
A \textbf{pattern} over $A$ is a member of $A^S$ for some finite $S \subset \mathbb{Z}^d$, which is said to have \textbf{shape} $S$. When
$d = 1$ and $S$ is an interval of integers, we use the term \textbf{word} rather than pattern.
\end{definition}

For any set $\mathcal{F}$ of patterns over $A$, one can define the set $X(\mathcal{F}) := \{x \in A^{\mathbb{Z}^d} \ : \ x(S) \notin \mathcal{F} \ \forall \textrm{ finite } S \subset \mathbb{Z}^d\}$. It is well known that any $X(\mathcal{F})$ is a $\mathbb{Z}^d$-subshift, and all $\mathbb{Z}^d$-subshifts are representable in this way. All subshifts are assumed to be nonempty in this paper.

For any patterns $v \in A^S$ and $w \in A^T$ with $S \cap T = \varnothing$, define $vw$ to be the pattern in $A^{S \cup T}$ defined by $(vw)(S) = v$ and $(vw)(T) = w$.

\begin{definition}
A \textbf{$\mathbb{Z}^d$-shift of finite type (SFT)} is a $\mathbb{Z}^d$-subshift equal to $X(\mathcal{F})$ for some finite $\mathcal{F}$. The \textbf{type} of $X$ is defined to be the minimum integer $t$ so that $\mathcal{F}$ can be chosen with all patterns on shapes which are subsets of $[1,t]^d$.
\end{definition}

Throughout this paper, for $a < b \in \mathbb{Z}$, $[a,b]$ will be used to denote $\{a,\ldots,b\}$, except for the special case
$[0,1]$, which will have its usual meaning as an interval of real numbers.

\begin{definition} 
The \textbf{language} of a $\mathbb{Z}^d$-subshift $X$, denoted by $L(X)$, is the set of all patterns which appear in points of $X$. For any finite $S \subset \mathbb{Z}^d$, $L_S(X) := L(X) \cap A^S$, the set of patterns in the language of $X$ with shape $S$.
\end{definition}

The following definitions are from \cite{pavlovTCPE} and relate to the conditions given there characterizing ZTCPE.

\begin{definition}
For any $\mathbb{Z}^d$-subshift $X$ and any finite $S \subseteq \mathbb{Z}^d$, patterns $w,w' \in L_S(X)$ are \textbf{exchangeable} in $X$ if there exist homoclinic points $x,x' \in X$ such that $x(S) = w$ and $x'(S) = w'$.
\end{definition}

It should be reasonably clear that if $X$ is an $\mathbb{Z}^d$-SFT with type $t$, then $w,w'$ are exchangeable if and only if there exists $N$ and $\delta \in L_{[-N,N]^d \setminus [-N+t,N-t]^d}(X)$ such that $\delta w, \delta w' \in L(X)$. 

\begin{definition}
For any $\mathbb{Z}^d$-subshift $X$ and any finite $S \subseteq \mathbb{Z}^d$, patterns $w,w' \in L_S(X)$ are \textbf{chain exchangeable} in $X$ if there exists $n$ and patterns $(w_i)_{i=1}^n$ in $L_S(X)$ such that $w_1 = w$, $w_n = w'$, and $w_i$ and $w_{i+1}$ are exchangeable in $X$ for $i \in [1,n)$.
\end{definition}

Alternately, the chain exchangeability relation is just the transitive closure of the exchangeability relation.

\begin{definition}
The \textbf{topological entropy} of a $\mathbb{Z}^d$ topological dynamical system $(X, T_v)$ is given by 
\[
h(X, T_v) := \sup_{\mathcal{U}} \lim_{n \rightarrow \infty} \frac{1}{n^d} N\left(\bigvee_{v \in [1,n]^d} T_v \mathcal{U} \right).
\]
where $\mathcal{U}$ ranges over open covers of $X$ and $N(\mathcal{U})$ is the minimal size of a subcollection of $\mathcal{U}$
which covers $X$.
\end{definition}

We will not need any advanced properties of topological entropy in this paper. (For a detailed treatment of topological entropy, see \cite{walters}.) We do, however, note the following sufficient condition for positive topological entropy. If $K$ and $K'$ are disjoint nonempty closed sets in $X$, and if there exists a subset $S$ of $\mathbb{Z}^d$ with positive density so that for any $y \in \{0,1\}^S$, there exists $x \in X$ with $T_s x \in K$ when $y(s) = 0$ and $T_s x \in K'$ when $y(s) = 1$, then it follows that $h(X, T_v) > 0$; in particular, for $\mathcal{U} = \{K^c, K'^c\}$, the limit in the definition is at least $\log 2$ times the density of $S$. For brevity, we refer to this property by saying that $(X, T_v)$ contains points which ``independently visit $K$ and $K'$ in any predetermined way along a set of iterates of positive density.''


\begin{definition}
A (topological) \textbf{factor map} is any continuous shift-commuting map $\phi$ from a $\mathbb{Z}^d$ topological dynamical system $(X, T_v)$ to a $\mathbb{Z}^d$ topological dynamical system $(Y, S_v)$. Given such a factor map $\phi$, the system $(\phi(X), S_v)$ is called a \textbf{factor of} $(X, T_v)$. 
\end{definition}


It is well-known that topological entropy does not increase under factor maps; again, see \cite{walters} for a proof.

\begin{definition}
A $\mathbb{Z}^d$ topological dynamical system $(X, T_v)$ has \textbf{topologically completely positive entropy} (or \textbf{TCPE}) if for every surjective factor map from $(X, T_v)$ to a $\mathbb{Z}^d$ topological dynamical system $(Y, S_v)$, either $h(Y, S_v) > 0$ or $|Y| = 1$.
\end{definition}

\begin{definition}
A $\mathbb{Z}^d$ topological dynamical system $(X, T_v)$ has \textbf{zero-dimensional topologically completely positive entropy} (or \textbf{ZTCPE}) if for every surjective factor map from $(X, T_v)$ to a $\mathbb{Z}^d$ zero-dimensional topological dynamical system $(Y, S_v)$, either 
$h(Y, S_v) > 0$ or $|Y| = 1$.
\end{definition}

Our final set of definitions relates to so-called balanced sequences. For any word $w$ on $\{0,1\}$, we use $\#(w,1)$ to denote the number of $1$ symbols in $w$.

\begin{definition}
A sequence $x \in \{0,1\}^{\mathbb{Z}}$ is \textbf{$k$-balanced} if every two subwords of $x$ of the same length have numbers of $1$ symbols within $k$, i.e. if for every $n,i,j$, $|\#(w([i, i+n-1]),1) - \#(w([j, j+n-1]),1)| \leq k$. We use simply the term \textbf{balanced} to mean $1$-balanced.
\end{definition}

The following lemma and corollary are standard; see for instance Chapter 2 of \cite{lothaire} for proofs in the $1$-balanced case which trivially extend to arbitrary $k$.

\begin{lemma}\label{balfreq}
For every $k$-balanced sequence $x$, there is a uniform frequency of $1$s, i.e. there exists $\alpha \in [0,1]$ so that for every 
$\epsilon > 0$, there exists $N$ such that for $n > N$, every $n$-letter subword of $x$ has proportion of $1$ symbols between $\alpha - \epsilon$ and $\alpha + \epsilon$. 
\end{lemma}

\begin{corollary}\label{within1}
For every $k$-balanced sequence $x$ with frequency $\alpha$, every $n \in \mathbb{N}$, and every $i \in \mathbb{Z}$, 
$|n\alpha - \#(x([i,i+n-1]),1)| \leq k$. 
\end{corollary}


For convenience, we refer to the uniform frequency of $1$s in a $k$-balanced sequence as its \textbf{slope}. The following is immediate. 

\begin{corollary}\label{bal2choice}
For any balanced sequence $x$ with slope $\alpha$ and $n \in \mathbb{N}$, if $n \alpha \notin \mathbb{Z}$, then for every 
$i \in \mathbb{Z}$, $\#(x([i,i+n-1]),1)$ is either $\lfloor n\alpha \rfloor$ or $\lceil n\alpha \rceil$.
\end{corollary}

Here are two examples of simple algorithmically generated balanced sequences; see Chapter 2 of \cite{lothaire} for a proof that they are in fact balanced with slope $\alpha$. 

\begin{definition}
For any $\alpha \in [0,1]$, the \textbf{lower characteristic sequence} $\underline{x_{\alpha}}$ is defined by
$\underline{x_{\alpha}}(n) = \lfloor (n+1) \alpha \rfloor - \lfloor n\alpha \rfloor$ for all $n \in \mathbb{Z}$.
The \textbf{upper characteristic sequence} $\overline{x_{\alpha}}$ is defined by
$\overline{x_{\alpha}}(n) = \lceil (n+1) \alpha \rceil - \lceil n\alpha \rceil$ for all $n \in \mathbb{Z}$.
\end{definition}

The lower and upper characteristic sequences are not shifts of each other for irrational $\alpha$,
but for rational $\alpha$ we note that they are. If we write $\alpha = \frac{i}{j}$ in lowest terms, then there exists
$k \in \mathbb{N}$ so that $k\alpha = m + \frac{1}{j}$ for an integer $m$. Then,
\begin{multline*}
\sigma^k(\overline{x_{\alpha}}(n)) = \overline{x_{\alpha}}(k + n) = \lceil (k + n + 1) \alpha \rceil - \lceil (k + n) \alpha \rceil \\
= m + (\lceil (n + 1) \alpha + \frac{1}{j} \rceil) - m - (\lceil n \alpha + \frac{1}{j} \rceil) = 
\lfloor (n + 1) \alpha \rfloor - \lfloor n \alpha \rfloor = \underline{x_{\alpha}}(n).
\end{multline*}
(We here used the easily checked fact that for any rational $x$ with denominator $j$, $\lceil x + \frac{1}{j} \rceil = 1 + \lfloor x \rfloor$.)

Characteristic sequences also have useful convergence properties.

\begin{lemma}\label{charconv}
If $\alpha_n$ approaches a limit $\alpha$ from above, then the lower characteristic sequences
$\underline{x_{\alpha_n}}$ converge to the lower characteristic sequence $\underline{x_{\alpha}}$. 
Similarly, if $\alpha_n$ approaches $\alpha$ from below, then the upper characteristic sequences
$\overline{x_{\alpha_n}}$ converge to the upper characteristic sequence $\overline{x_{\alpha}}$.
\end{lemma}

\begin{proof}
This follows immediately from the continuity of the floor function from the right and the continuity of the ceiling 
function from the left.
\end{proof}

For irrational $\alpha$, the structure of balanced sequences is well-known; the set of such balanced sequences
is just the so-called Sturmian subshift with rotation number $\alpha$. We omit a full treatment of
Sturmian sequences here and instead refer the reader to Chapter 2 of \cite{lothaire} for a detailed analysis.
We will say that Sturmian sequences are defined similarly to upper and lower characteristic sequences,
with the change that one is also allowed to add any constant to the terms inside floor or ceiling functions 
(e.g. $x$ defined by $x(n) = \lceil \pi + (n + 1) \alpha \rceil - \lceil \pi + n\alpha \rceil$ is Sturmian). 

For rational $\alpha$, the structure of balanced sequences is more complicated. All balanced sequences
with rational slope are eventually periodic (Proposition 2.1.11, \cite{lothaire}). The periodic
balanced sequences are easy to describe; the following is essentially Lemma 2.1.15 from \cite{lothaire},
combined with the above observation that upper and lower characteristic sequences are shifts of each 
other for rational $\alpha$.

\begin{lemma}\label{balper}
Every balanced sequence with rational slope $\alpha$ which is periodic is a shift of
$\underline{x_{\alpha}}$.
\end{lemma}

The eventually periodic but not periodic balanced sequences are more complicated, they are described
as ``skew sequences'' in \cite{morsehedlund}. Luckily we do not need a complete description
of such sequences in this work, but we will need the following useful fact, stated as Proposition 2.1.17 in \cite{lothaire}.

\begin{lemma}\label{balsturm}
Every balanced sequence can be written as the limit of balanced sequences with irrational slopes, i.e. Sturmian sequences.
\end{lemma}

\section{Proofs}\label{proofs}

We first establish the claim from the introduction that ``subshift TCPE'' is in fact the same as ZTCPE. 

\begin{theorem}\label{STCPE=ZTCPE}
A $\mathbb{Z}^d$ topological dynamical system has a nontrivial $\mathbb{Z}^d$-subshift factor with zero entropy if and only if it has a nontrivial zero-dimensional factor with zero entropy.
\end{theorem}

\begin{proof}

The forward direction is trivial, so we prove only the reverse. Suppose that $(X,T_v)$ is a topological dynamical system with a factor $(Y,S_v)$ where $|Y| > 1$, $Y$ is zero-dimensional, and $h(Y,S_v) = 0$.
Then, since $|Y| > 1$ and $Y$ is zero-dimensional, there exists a nontrivial partition of $Y$ into clopen sets $A$ and $B$. Then, define the map $\phi: Y \rightarrow \{0,1\}^{\mathbb{Z}^d}$ as follows:
$(\phi(y))(v) = \chi_B(S_v(y))$, i.e. $(\phi(y))(v) = 0$ if $S_v(y) \in A$ and $(\phi(y))(v) = 1$ if $S_v(y) \in B$. Since $A$ and $B$ are closed, $\phi$ is a surjective factor map from $(Y,S_v)$ to the subshift $(\phi(Y), \sigma_v)$. Moreover, if $a \in A$ and $b \in B$, then $(\phi(a))(0) = 0$ and $(\phi(b))(0) = 1$, meaning that $|\phi(Y)| > 1$. Therefore, $(\phi(Y), \sigma_v)$ is a nontrivial subshift factor of $(X,T_v)$, and it has zero entropy since it is a factor of the zero entropy system $(Y, S_v)$.

\end{proof}

Theorem~\ref{1dSFT} is a corollary of well-known results, but for completeness we supply the simple proof here.

\begin{proof}[Proof of Theorem~\ref{1dSFT}]

We assume basic knowledge of the structure of $\mathbb{Z}$-SFTs; for more information, see \cite{LM}. 

Consider a $\mathbb{Z}$-SFT $X$, which without loss of generality we may assume to be nearest-neighbor. If $X$ is not mixing, then 
it is either reducible or periodic. If $X$ is reducible, then the factor map which carries each letter to its
irreducible component has (zero-dimensional) image which is a nontrivial (there are at least two irreducible components) SFT 
given by a directed graph with no cycles, thereby of zero entropy. If $X$ is periodic, then the factor map which carries each letter to its period class has image which is a nontrivial (and zero-dimensional) finite union of periodic orbits, thereby of zero entropy. We have shown that any $\mathbb{Z}$-SFT with ZTCPE is mixing. 

Then, it is well-known that a $\mathbb{Z}$-SFT is mixing if and only if it has the specification property, which clearly 
implies TCPE since specification is preserved under factors, and every nontrivial dynamical system with specification has
positive entropy. 

\end{proof}

\begin{proof}[Proof of Theorem~\ref{TCPEthm}]

We assume some familiarity with the proof of Theorem~\ref{ZTCPEthm} from \cite{pavlovTCPE}, and so only summarize the required changes. Suppose that $X$ is a $\mathbb{Z}^d$-SFT of type $t$ with the properties that every $w \in L(X)$ has positive measure for some $\mu \in \mathcal{M}(X)$ and that there exists $N$ so that for all $S \subseteq \mathbb{Z}^d$ and $w,w' \in L_S(X)$, there exist $w = w_1, w_2, \ldots, w_N = w'$ so that for every $i \in [1,N)$, there are homoclinic points in $[w_i]$ and $[w_{i+1}]$. Now, consider any surjective factor map $\phi: (X, \sigma_v) \rightarrow (Y, S_v)$ with $|Y| > 1$. Since $|Y| > 1$, there exist $y,y' \in Y$ with
$d_Y(y, y') = \alpha > 0$. By uniform continuity of $\phi$, there exists $\delta > 0$ so that $d_X(x,x') < \delta \Longrightarrow d_Y(y,y') < \frac{\alpha}{N}$. Choose $n$ so that the cylinder set of any $w \in L_{[-n,n]^d}(X)$ has diameter less than $\delta$. 

Choose $x \in \phi^{-1}(y)$ and $x' \in \phi^{-1}(y')$, and define $w = x([-n,n]^d)$ and $w' = x'([-n,n]^d)$. Then by assumption, there exist $w = w_1, w_2, \ldots$ \newline $w_N = w'$ with the above described properties. Note that each $\phi([w_i])$ has diameter less than $\frac{\alpha}{N}$, $y \in \phi([w_1])$, $y' \in \phi([w_N])$, and $d(y, y') = \alpha$. This implies that there exists $i$ for which $\phi([w_i])$ and $\phi([w_{i+1}])$ are disjoint closed subsets of $Y$. From here, the proof proceeds essentially as in \cite{pavlovTCPE}; we again will only briefly summarize. Firstly, since there exist homoclinic points in $[w_i]$ and $[w_{i+1}]$, there exists a boundary pattern $\delta$ of thickness $t$ which can be filled with either $w_i$ or $w_{i+1}$ at the center. By assumption, there exists $\mu \in \mathcal{M}(X)$ with $\mu([\delta]) > 0$ and therefore a point $x \in X$ with a positive frequency of occurrences of $\delta$. Then, since $X$ is an SFT with type $t$, each occurrence of $\delta$ in $x$ can be independently filled with either $w_i$ or $w_{i+1}$ at the center. The $\phi$-images of this family of points then visit the disjoint closed sets $\phi([w_i])$ and $\phi([w_{i+1}])$ under $S_v$ in any predesignated way along a set of $v \in \mathbb{Z}^d$ of positive density, which is enough to imply positive entropy of $(Y, S_v)$.

\end{proof}

In particular, several existing proofs in the literature which purported to prove TCPE while in truth only verifying ZTCPE can be shown to actually yield TCPE via Theorem~\ref{TCPEthm}.

\begin{corollary}\label{hochmanTCPE}
The topologically mixing $\mathbb{Z}^2$-SFT defined in Section 6.3 of \cite{hochman} has TCPE.
\end{corollary}

\begin{proof}
It was shown in \cite{BPS} that any two patterns in the example in question are exchangeable, i.e. that the hypotheses of Theorem~\ref{ZTCPEthm} are satisfied with $n = 2$. Therefore, Theorem~\ref{TCPEthm}
implies TCPE.
\end{proof}

\begin{corollary}\label{bgTCPE}
Every nontrivial block gluing $\mathbb{Z}^d$-SFT has TCPE. 
\end{corollary}

\begin{proof}
Theorem 1.5 from \cite{pavlovTCPE} shows that any two patterns in a block gluing $\mathbb{Z}^d$-SFT are exchangeable, i.e. that the hypotheses of Theorem~\ref{ZTCPEthm} are satisfied with $n = 2$. Therefore, Theorem~\ref{TCPEthm}
implies TCPE.
\end{proof}

\begin{corollary}\label{ribbonTCPE}
The $\mathbb{Z}^2$-SFT from Examples 1.2 and 1.3 of \cite{pavlovTCPE} has TCPE.
\end{corollary}

\begin{proof}
It is shown in \cite{pavlovTCPE} that the example in question satisfies the hypotheses of Theorem~\ref{ZTCPEthm} with $n = 3$. Therefore, Theorem~\ref{TCPEthm} implies TCPE.
\end{proof}

Our main tool for constructing the examples of Theorems~\ref{1dsub} and \ref{2dSFT} is the following ``black box'' which, given an input subshift $X$ with some very basic transitivity properties, yields a subshift with TCPE. Though the technique should work in any dimension, for brevity we here restrict ourselves to $d \leq 2$.

\begin{theorem}\label{blackbox}
For $d \leq 2$ and any alphabet $A$, there exists an alphabet $B$ and a map $f$ taking any orbit of a point in $A^{\mathbb{Z}^d}$ to a union of orbits of points in $B^{\mathbb{Z}^d}$ with the following properties.\\

\noindent
{\rm (1)} $O(x) \neq O(x') \Longrightarrow f(O(x)) \cap f(O(x')) = \varnothing$.

\noindent
{\rm (2)} For any $\mathbb{Z}^d$-subshift (resp. SFT) $X$, $f(X)$ is a $\mathbb{Z}^d$-subshift (resp. SFT).

\noindent
{\rm (3)} If $X$ is a $\mathbb{Z}^d$-subshift with the following two properties:

{\rm(3a)} Every $w \in L(X)$ has positive measure for some $\mu \in \mathcal{M}(X)$

{\rm(3b)} there exists $N$ so that for every $w, w' \in L_{[-n,n]^d}(X)$, there exist 

patterns $w = w_1, w_2, \ldots, w_N = w'$ in $L_{[-n,n]^d}(X)$ so that for all

$i \in [1,N)$, $w_i$ and $w_{i+1}$ coexist in some point of $X$,

\noindent
then $f(X)$ has TCPE. 

\end{theorem}

\begin{proof}
Before beginning the proof, we note that by the pointwise ergodic theorem, (3a) is clearly equivalent to the statement that for every $w$, $X$ contains a point with a positive frequency of occurrences of $w$. Where it is useful, we prove/use this equivalent version without further comment.

We first deal with $d = 1$, which is a significantly easier proof and shows the ideas required for the more difficult $d = 2$ case. Consider any alphabet $A$, take a symbol $0 \notin A$, and to any orbit $O(x)$, define $f(x)$ to be the set of all points in $(A \cup 0)^{\mathbb{Z}}$ of the form $\ldots a_{-1} 0^{n_{-1}} a_0 0^{n_0} a_1 0^{n_1} \ldots$, where $\ldots a_{-1} a_0 a_1 \ldots \in O(x)$ and each $n_i$ is $2$, $3$, or $4$. 

It is easily checked that $f(x)$ is shift-invariant, i.e. a union of orbits. Clearly, if some point $\ldots a_{-1} 0^{n_{-1}} a_0 0^{n_0} a_1 0^{n_1} \ldots$ is in $f(O(x)) \cap f(O(x'))$, then $\ldots a_{-1} a_0 a_1 \ldots \in O(x) \cap O(x')$, verifying (1).

If $X$ is a $\mathbb{Z}$-subshift defined by a set $\mathcal{F}$ of forbidden words, then the reader may check that $f(X)$ is a $\mathbb{Z}$-subshift defined by the forbidden list
\begin{multline*}
\mathcal{F'} = \{00000\} \cup \{ab \ : \ a,b \in A\} \cup \{a0b \ : \ a,b \in A\} \\
\cup \{w_1 0^{n_1} w_2 0^{n_2} \ldots 0^{n_{k-1}} w_k \ : \ n_i \in \{2,3,4\}, w_1 w_2 \ldots w_n \in \mathcal{F}\}.
\end{multline*}
If $X$ is an SFT, then $\mathcal{F}$ can be chosen to be finite, in which case $\mathcal{F'}$ is also finite, showing that $f(X)$ is an SFT and verifying (2).

It remains only to show (3). We begin with some notation. For any point $y \in f(X)$, there exists $x$ for which $y \in f(x)$, and 
by (1), $x$ is uniquely determined up to shifts. We say for any such $x$ that $y$ is \textbf{induced by} 
$x$. Similarly, for any finite word $w \in L(f(X))$, the non-zero letters of $w$ (in the same order) form a word $v \in L(X)$, and we say that $w$ is \textbf{induced by} $v$. Suppose that $X$ satisfies (3a) and (3b), and consider any surjective factor map $\phi: (f(X), \sigma) \rightarrow (Y, S)$ with $|Y| > 1$. By considering the words in $L(X)$ ``inducing'' arbitrary words $w,w' \in L(f(X))$, it is easily checked that (3a) and (3b) hold for $f(X)$ as well.

Since $N$ does not depend on the words chosen in (3b), we can proceed as in the proof of Theorem~\ref{TCPEthm} to find words $w_i, w_{i+1} \in L(f(X))$ of the same length $L$ for which $\phi([w_i])$ and $\phi([w_{i+1}])$ are disjoint closed sets, and $w_i$ and $w_{i+1}$ coexist in some word $u \in L(f(X))$. We can assume without loss of generality that $u$ begins and ends with non-zero letters by extending it slightly on the left and right, and define 
$v \in L(X)$ which induces $u$. We fix single occurrences of $w_i$ and $w_{i+1}$ within $u$, assume without loss of generality that $w_i$ appears to the left of $w_{i+1}$, and denote by $k$ the horizontal distance 
between them. 

Since $v \in L(X)$, we may choose $p, s \in L_k(X)$ so that $pvs \in L(X)$. Then, we define the following words:
\[
u' = p(1) 0^3 p(2) 0^3 \ldots 0^3 p(k) 0^3 u 0^3 s(1) 0^3 s(2) 0^3 \ldots 0^3 s(k) \textrm{ and}
\]
\[
u'' = p(1) 0^2 p(2) 0^2 \ldots 0^2 p(k) 0^2 u 0^4 s(1) 0^4 s(2) 0^3 \ldots 0^4 s(k).
\]

Since both $u'$ and $u''$ are induced by $pvs \in L(X)$, $u'$ and $u''$ are both in $L(f(X))$. They also have the same length. In addition, since $u''$ is created by reducing the first $k$ gaps of $0$s in $u'$ by one and increasing the last $k$ gaps of $0$s in $u'$ by one, the occurrence of $w_{i+1}$ in the central $u$ of $u'$ occurs at a location $k$ units further to the left within $u''$. In other words, there exists $j$ so that $u'([j,j + L - 1]) = w_i$ and $u''([j, j + L - 1]) = w_{i+1}$. 

Since $u' \in L(f(X))$, by (3a) there exists $y \in f(X)$ containing a positive frequency of occurrences of $u'$. The rules defining $f(X)$ should make it clear that any subset of these occurrences of $u'$ can be replaced by $u''$ to yield a collection of points of $f(X)$. Then, as before, the image under $\phi$ of this collection yields a collection of points of $Y$ which independently visit the disjoint closed sets $\phi([w_i])$ and $\phi([w_{i+1}])$ under $S$ in any predesignated way along a set of iterates of positive density, proving that $h(Y,S) > 0$ and completing the proof of (3).\\

Now, we must describe $f$ and prove (1)-(3) for $d = 2$ as well. Many portions of the argument are quite similar, and so we will only comment extensively on the portions which require significantly more details. First, we describe auxiliary $\mathbb{Z}^2$ shifts of finite type $X_H$ and $X_V$ which will help with the definition of $f$. The alphabet for $X_H$ is $\{0,H\}$, and the SFT rules are as follows.\\

\noindent
$\bullet$ Each column consists of $H$ symbols separated by gaps of $0$ symbols with lengths $2$, $3$, or $4$.

\noindent
$\bullet$ Given any $H$ symbol, exactly two of its neighbors (in cardinal directions) are $H$ symbols.

\noindent
$\bullet$ If two $H$ symbols are diagonally adjacent, they must have exactly one $H$ symbol as a common neighbor.

\noindent
$\bullet$ Three $H$ symbols may not comprise a vertical line segment.

\noindent
$\bullet$ If three $H$ symbols comprise a diagonal line segment, then the central of the three must have $H$ symbols to its left and right. (For instance, $\begin{smallmatrix} H & & \\ H & H & H\\ & & H \end{smallmatrix}$ is legal, but $\begin{smallmatrix} H & H & \\ & H & H\\ & & H \end{smallmatrix}$ is not.)\\

The reader may check these rules make $X_H$ a $\mathbb{Z}^2$-SFT of type $5$. Points of $X_H$ consist of biinfinite meandering ribbons of $H$ symbols, which either move up one unit, down one unit, or stay at the same height for each unit moved to the right or left, and which may not ``meander'' twice consecutively in the same direction. Informally, this means that any horizontal ribbon has ``slope'' with absolute value less than or equal to $\frac{1}{2}$. Every point of $X_H$ contains infinitely many such ribbons, and any pair of closest ribbons may not touch diagonally and are always separated by a vertical gap of $0$ symbols of length either $2$, $3$, or $4$; see Figure~\ref{ztcpepic1}.

\begin{figure}[h]
\centering
\includegraphics[scale=0.3]{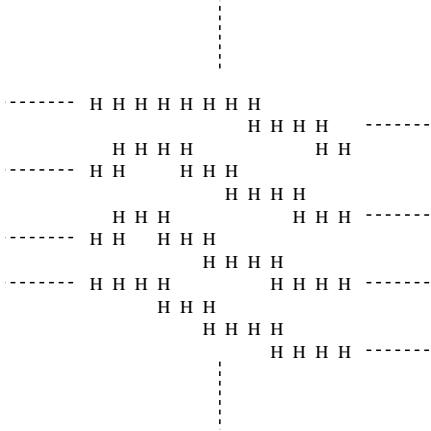}
\caption{Part of a point of $X_H$}
\label{ztcpepic1}
\end{figure}

We also define the $\mathbb{Z}^2$-SFT $X_V$ with alphabet $\{0,V\}$ with vertical rather than horizontal ribbons, where legal points are just legal points of $X_H$, rotated by ninety degrees, with $H$ symbols replaced by $V$ symbols. In particular,  vertical ribbons have ``slope'' with absolute value at least $2$. We will require the following fact about $X_H$.\\

\noindent
\textbf{Claim A1:} Every pattern $w \in L(X_H)$ appears in a point $x \in X_H$ homoclinic to the point $x_0 \in X_H$ consisting of flat horizontal ribbons, equispaced by $3$ units, one of which passes through the origin.

\begin{proof}

Our proof is quite similar to a proof given in \cite{pavlovTCPE} for a slightly different system, and so for brevity we do not include every technical detail here. It clearly suffices to only treat $w \in L_{[-n,n]^d}(X_H)$ for some $n$. We proceed in three steps.\\

\noindent
Step 1: Complete each horizontal ribbon segment in $w$ to create $w'$ in which each horizontal ribbon segment touches the infinite vertical lines given by the left and right edges of $w$.\\

\noindent
Step 2: Allow the ribbons in $w'$ to meander on the left and right until they are equispaced with distance $3$, each ribbon has the same height at the left and right edge, and those heights are the same as those of ribbons in $x_0$, i.e. multiples of $4$.\\

\noindent
Step 3: Place additional ribbons above, one at a time, with left and right edges equispaced with distance $3$, each of which ``unravels'' the leftmost meandering in the ribbon below, until arriving at a completely horizontal ribbon; then continue with infinitely many more completely horizontal ribbons equispaced with distance $3$. Perform a similar procedure below. (See Figure~\ref{ztcpepic3} for an illustration.)\\

\begin{figure}[h]
\centering
\includegraphics[scale=0.3]{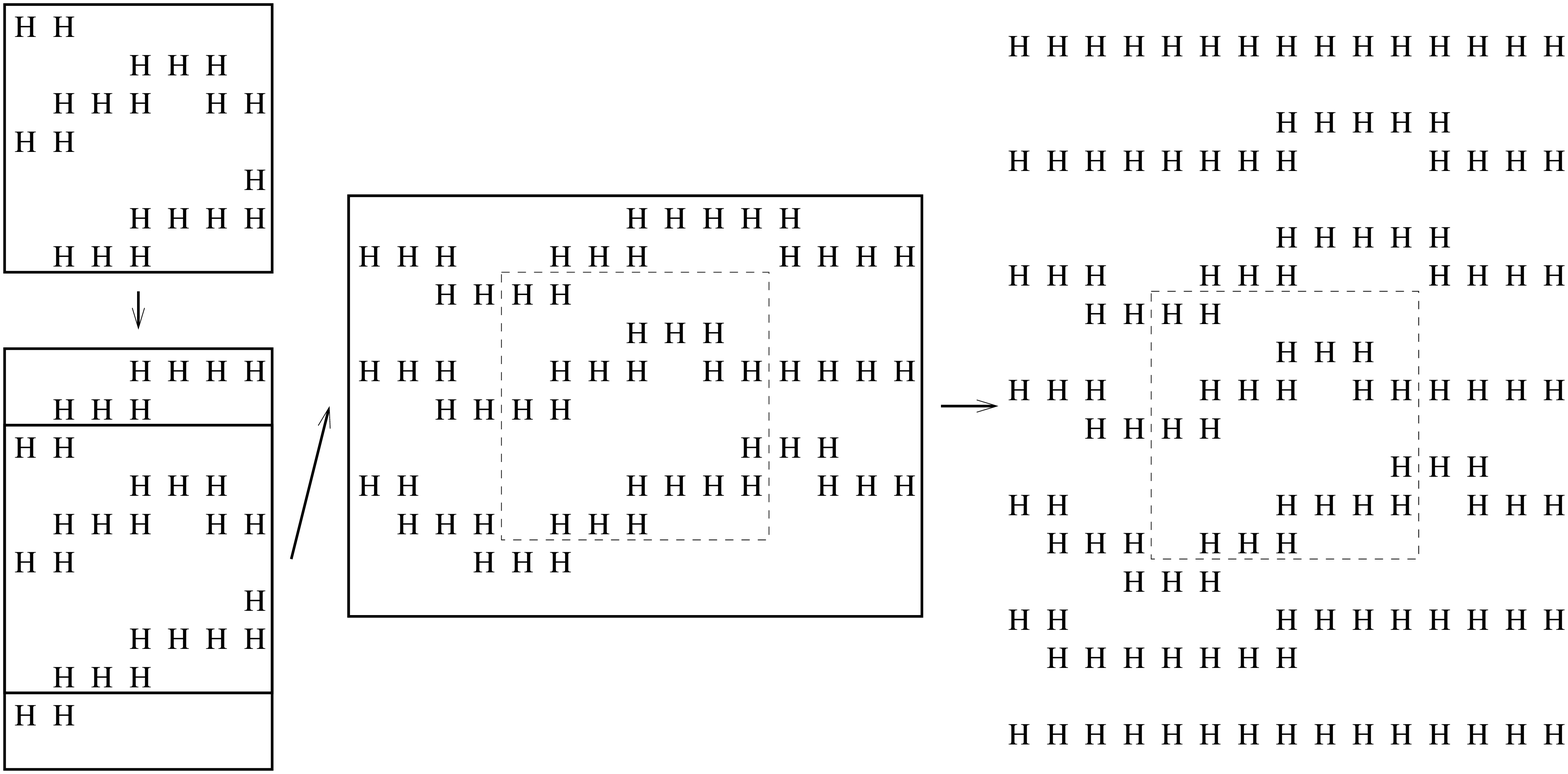}
\caption{The steps of embedding $w$ in a point $x \in X_H$ homoclinic to $x_0$}
\label{ztcpepic3}
\end{figure}

The resulting point $x \in X_H$ is clearly homoclinic with $x_0$ and contains $w$, completing the proof.

\end{proof}

Every point of $X_H$ must contain infinitely many ribbons; in a point of $X_H$, we index these by $\mathbb{Z}$, beginning with the $0$th as the first encountered when beginning from the origin and moving straight up, and then proceeding with the positively-indexed ribbons above it and negatively-indexed ribbons below it. Similarly, the vertical ribbons of a point of $X_V$ are indexed starting at the $0$th ribbon being the first encountered by beginning from the origin and moving to the right, positively-indexed ribbons to the right, and negatively-indexed ribbons to the left. By the earlier noted restrictions on slopes, any horizontal ribbon and any vertical ribbon must intersect at either a single site or a pair or triple of adjacent (including diagonals) sites (see Figure~\ref{ztcpepic2}), and so for any pair of points $x \in X_H$ and $y \in X_V$, we can assign an injection $f_{x,y}$ from $\mathbb{Z}^2$ to itself by defining $f_{x,y}(i,j)$ to be the lexicographically least site within the intersection of the $i$th horizontal ribbon and $j$th vertical ribbon.

We are now ready to define $f$. For any alphabet $A$ and orbit $O(x)$ in $A^{\mathbb{Z}^2}$, choose a symbol $0 \notin A$ and define $f(x)$ to be the collection of all points $z$ on the alphabet $B = \{(0,0,0), (0,V,0), (H,0,0)\} \cup (\{(H,V)\} \times (A \sqcup \{0\}))$ with the following properties:\\

\noindent
$\bullet$ The first coordinate of $z$ is a point $x$ of $X_H$

\noindent
$\bullet$ The second coordinate of $z$ is a point $y$ of $X_V$

\noindent
$\bullet$ Letters of $A$ may only appear in the third coordinate at the lexicographically least sites within intersections of ribbons from $x$ and $y$. If we define $t \in A^{\mathbb{Z}^2}$ by taking $t(i,j)$ to be the letter of $A$ in the third layer at $z(f_{x,y}(i,j))$, then $t \in O(x)$.\\

\begin{figure}[h]
\centering
\includegraphics[scale=0.4]{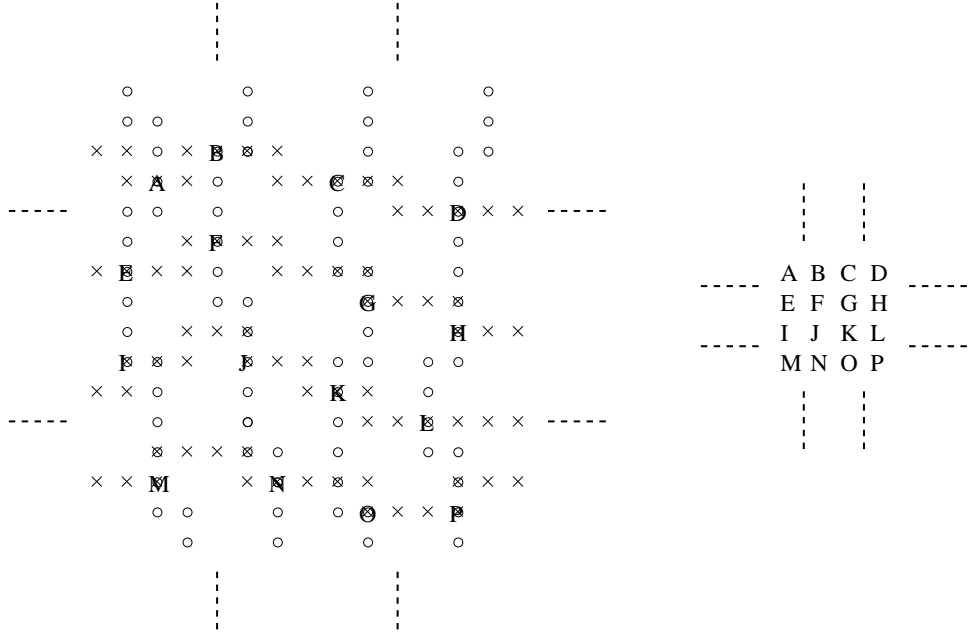}
\caption{A point of $f(X)$ and the pattern in $L(X)$ given by the letters of $A$ at (lexicographically minimal sites within) its ribbon intersections; $\times$ and $\circ$ here represent $H$ and $V$ respectively}
\label{ztcpepic2}
\end{figure}

We claim now that this $f$ has the desired properties. It is easily checked that $f(x)$ is shift-invariant, i.e. a union of orbits. Clearly, if $f(O(x))$ and $f(O(x'))$ share a point $z$, then $t$ defined as in the third bullet point above must be in $O(x) \cap O(x')$, verifying (1).

An explicit description of a forbidden list inducing $f(X)$ for a $\mathbb{Z}^2$-subshift $X$ would be needlessly long and complicated; instead we just informally describe the restrictions. Firstly, a finite forbidden list can be used to force the first and second coordinates of any point in $f(X)$ to be in $X_H$ and $X_V$ respectively. Then, since intersections of ribbons are finite sets of adjacent or diagonally adjacent sites which cannot be adjacent or diagonally adjacent to each other, a finite forbidden list can force letters of $A$ to occur on the third coordinate precisely at lexicographically minimal sites within ribbon crossings. Finally, one must only choose a forbidden list $\mathcal{F}$ which induces $X$ and forbid all finite patterns whose first and second coordinates form legal patterns in $X_H$ and $X_V$, but whose third coordinate contains a pattern from $\mathcal{F}$ on the lexicographically least sites within intersections of ribbons from the first two. Then $f(X)$ is a subshift, and again it should be clear that if $X$ is an SFT, then $\mathcal{F}$ can be chosen finite, yielding a finite forbidden list for $f(X)$, implying that $f(X)$ is an SFT and completing the proof of (2).

It remains to prove (3). Choose any $X$ satisfying (3a) and (3b), and again we begin with notation: any point $y \in f(X)$ is \textbf{induced by} $x \in X$ if $y \in f(x)$, and by (1), $x$ is uniquely determined up to shifts.
For finite patterns, the geometry of the ribbons makes a similar definition trickier. For any $w \in L_S(f(X))$, choose $y \in f(X)$ with $w = y(S)$, define $a,b$ to be the first and second coordinates of $y$, define 
$T = f_{a,b}^{-1}(S)$, and define $v \in L_T(X)$ by taking $v(i,j)$ to be the third coordinate of $y(f_{a,b}(i,j)) = w(f_{a,b}(i,j))$; we say that $w$ is \textbf{induced by} $v$. We begin with the following auxiliary claim.\\

\noindent
\textbf{Claim A2:} $f(X)$ satisfies (3a) and (3b).

\begin{proof}

Choose any $v \in L_S(f(X))$, which is induced by $v' \in L_T(X)$. Since $X$ satisfies (3a), there is a point $x \in X$ with positive frequency of occurrences of $v'$; say $x(i + T) = v'$ for all $i \in I$ a subset of $\mathbb{Z}^2$ with positive density. Define 
$v'' \in L_S(X_H)$ and $v''' \in L_S(X_V)$ to be the restrictions of $v$ to its first and second coordinates respectively. By Claim A1 above, $v''$ appears within a point $x'' \in X_H$ homoclinic to the point $x_0$ of equispaced flat horizontal ribbons, and $v'''$ appears within a point of $x''' \in X_V$ homoclinic to the point $y_0$ of equispaced flat vertical ribbons. Then the pair $(v'', v''')$ appears within a point of $X_H \times X_V$ homoclinic to the doubly periodic point $(x_0, y_0)$ of equispaced horizontal and vertical ribbons. Finally, since $X_H \times X_V$ is an SFT, this means that
$(v'', v''')$ appears within some periodic point $(y_1, y_2) \in X_H \times X_V$, which then contains $(v'',v''')$ at a set of sites forming a coset $G$ of $\mathbb{Z}^2$ of finite index. Denote by $H$ the set $f_{y_1, y_2}^{-1}(G)$, which is also a finite index coset of $\mathbb{Z}^2$. Since $I$ has positive density, there exists $t \in \mathbb{Z}^2$ so that $H \cap (t + I)$ also has positive density. Construct a point $z \in f(X)$ by ``superimposing'' $\sigma_t x$ in the third coordinate at (lexicographically minimal sites within) intersections of ribbons in $y$. Then $v$ appears with positive frequency in $z$, proving (3a) for $f(X)$.

Now, consider any two patterns $v \neq w \in L_{[-n,n]^2}(f(X))$. As above, they are induced by $v' \in L_T(X)$ and $w' \in L_{T'}(X)$. We may extend $v'$ and $w'$ to patterns $t', u' \in L_{[-n,n]^2}(X)$
which induce $t,u \in L(f(X))$ containing $v,w$ respectively. Then, by assumption, there exist
$t' = w'_1, w'_2, \ldots, w'_N = u'$, all in $L_{[-n,n]^2}(X)$, 
so that for every $i$, $w'_i$ and $w'_{i+1}$ coexist in a point of $X$. We may in fact assume that both occur 
infinitely many times in the same point of $X$, since by (3a), any pattern containing 
both occurs with positive frequency in some point of $X$. 
In particular, $w'_i$ and $w'_{i+1}$ appear with arbitrarily large separation in some point of $X$.
Choose any patterns $t = w_1, w_2, \ldots, w_N = u$, all in $L(f(X))$, where each $w_i$ is induced by $w'_i$. 
Each $w_i$ then has shape containing $[-n,n]^2$ since the letters of $w'_i$ are ``stretched out'' to 
be placed within intersections of ribbons. 
For each $i$, define $w''_i$ to be the pattern given by the first two coordinates of $w_i$.
By Claim A1, for $i \in [1,N)$, for any large enough $v \in (4\mathbb{Z})^2$, 
we may place $w''_i$ and $w''_{i+1}$, separated by $v$, in some point of $X_H \times X_V$ homoclinic to the doubly periodic
point $(x_0, y_0)$. 

We can then create a point of $X$ containing $w'_i$ and $w'_{i+1}$ with large enough separation
that they may be superimposed over $w''_i$ and $w''_{i+1}$ in such a point of $X_H \times X_V$ to yield a point of $f(X)$ containing $w_i$ and $w_{i+1}$. Though the patterns $w'_i$ do not have the proper shape $[-n,n]^2$, each has shape containing $[-n,n]^2$, and we can pass to subpatterns with that shape (yielding $v$ from $w_1$ and $w$ from $w_N$ in particular) which still have the desired properties.
We have then shown that $f(X)$ satisfies (3b). 

\end{proof}

Now, consider any surjective factor map $\phi: (f(X), \sigma_v) \rightarrow (Y, S_v)$ with $|Y| > 1$. Again, as was done in the proof of Theorem~\ref{TCPEthm}, we can find patterns $w_i, w_{i+1} \in L_{[-n,n]^2}(f(X))$ for which 
$\phi([w_i])$ and $\phi([w_{i+1}])$ are disjoint closed sets, and $w_i$ and $w_{i+1}$ coexist in some pattern $u \in L(f(X))$. 

We fix single occurrences of $w_i$ and $w_{i+1}$ within $u$, denote by $t$ the vector pointing from $w_i$ to $w_{i+1}$ in $u$, and denote by $k$ the 
$\ell_1$-norm $|t_1| + |t_2|$ of $t$. Our goal is now to extend $u$ to a larger pair of patterns in $L(f(X))$ which contain $w_i$ and $w_{i+1}$ at the same location. We begin by defining a pair 
$u_1$ and $u'_1$ which contain occurrences of $w_i$ and $w_{i+1}$ respectively, separated by a vector $t_1$ with $\ell_1$-norm smaller than $k$. 

First, by Claim A1, we can extend $u$ to an entire point $y_1 \in f(X)$ whose first two coordinates (i.e. ``ribbon structure'') are homoclinic
to $x_0 \times y_0$, the point with equispaced horizontal and vertical ribbons. The (lexicographically minimal sites within) ribbon crossings of $y_1$ are filled as in some point $x \in X$ extending $u$. 

Then, we perturb $y_1$ to create a new point $y'_1$ in a way controlled by $t$. If the first coordinate of $t$ is nonzero, then we force all vertical ribbons which intersect the occurrence of $u$ within $y_1$ to meander a single unit to move the occurrence of $u$ within $y_1$ either left or right depending on whether the first coordinate of $t$ is positive or negative, respectively. The resulting point, which we call 
$y'_1$, is still in $f(X)$ since we changed no horizontal ribbons, did not change the $A$ letters at crossing points of ribbons, and the horizontal separation between vertical ribbons could only have been changed from $3$ to $2$ or $4$, both legal in $f(X)$. If it was the second coordinate of $v$ that was nonzero, then we force horizontal ribbons to meander to move the occurrence of $u$ within $y_1$ either down or up depending on whether the second coordinate of $v$ is positive or negative, respectively. (See Figure~\ref{ztcpepic4}.)

\begin{figure}[h]
\centering
\includegraphics[scale=0.4]{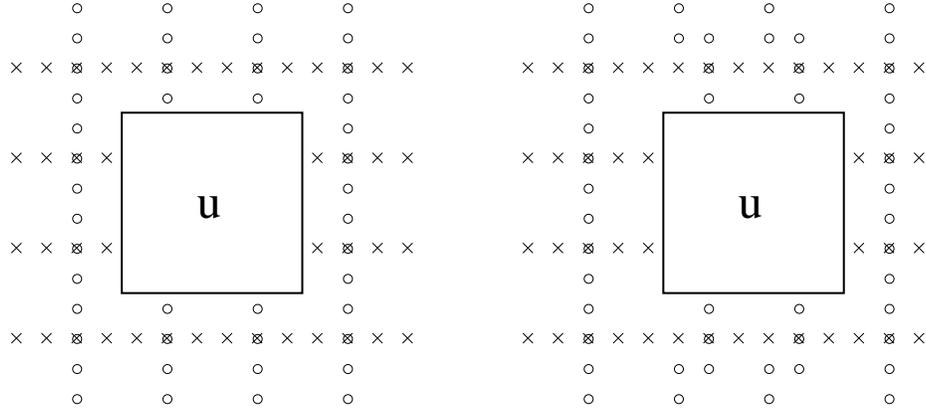}
\caption{Changing $y_1$ to $y'_1$ when the first coordinate of $t$ is negative; again $\times$ and $\circ$ represent $H$ and $V$ respectively}
\label{ztcpepic4}
\end{figure}

Since we only changed finitely many ribbons of $y_1$ at finitely many locations to create $y'_1$, $y_1$ and $y'_1$ are homoclinic. 
Clearly $u$ is still a subpattern of both $y_1$ and $y'_1$, and so we may restrict $y_1$ and $y'_1$ to some finite box to create patterns 
$u_1$ and $u'_1$ which are equal on their boundaries of thickness $5$. By the movement of the copy of $u$ within $u_1$ to create $u'_1$, there exists a vector $t_1$ with $\ell_1$ norm less than $k$ and a location $i_1$ so that $u_1(i_1 + S) = w_i$ and $u'_1(i_1 + t_1 + S) = w_{i+1}$. 

We now simply repeat this procedure finitely many times until arriving at $u_k$ and $u'_k$ in $L(f(X))$ which agree on their boundaries of thichness $t$ and for which there exists $i_k$ with $u_k(i_k + S) = w_i$ and $u'_k(i_k + S) = w_{i+1}$. 
Then, since $u_k \in L(f(X))$, by (3a) there is a point $y \in f(X)$ containing a positive frequency of occurrences of $u_k$. Since $u_k$ and $u'_k$ carry the same letters of $A$ at (lexicographically minimal sites within) ribbon intersections and have the same boundaries of thickness $5$ (the type of $X_H \times X_V$), the rules of $f(X)$ should make it clear that any subset of these occurrences of $u_n$ in $y$ can be replaced by $u'_n$ to yield a collection of points of $f(X)$. Then, as before, the image under $\phi$ of this collection yields a collection of points of $Y$ which independently visit the disjoint closed sets $\phi([w_i])$ and $\phi([w_{i+1}])$ under $S_v$ in any predetermined way along a set of positive density, proving that $h(Y,S_v) > 0$ and completing the proof of (3).\\

\end{proof}
We are now ready to prove Theorem~\ref{1dsub}. 

\begin{proof}[Proof of Theorem~\ref{1dsub}] 
For every $\alpha \in [0,1]$, denote by $B_\alpha$ the $\mathbb{Z}$-subshift consisting of all balanced sequences on $\{0,1\}$ with slope 
$\alpha$. (Recall that $[0,1]$ denotes the usual interval of real numbers, not the set $\{0,1\}$.) Then, define
\[
B = \bigsqcup_{\alpha \in [0,1]} B_\alpha,
\]
the $\mathbb{Z}$-subshift consisting of all balanced sequences on $\{0,1\}$. Then, by (2) of Theorem~\ref{blackbox}, $f(B)$ is a 
$\mathbb{Z}$-subshift, and by (1) of Theorem~\ref{blackbox}, it can be written as
\[
f(B) = \bigsqcup_{\alpha \in [0,1]} f(B_\alpha).
\]

\noindent
\textbf{Claim B1:} $(f(B), \sigma)$ does not have TCPE.

\begin{proof}
We will show that there is a surjective factor map from $(f(B), \sigma)$ to the nontrivial zero entropy system $([0,1], \textrm{id})$. The map $\pi$ is defined as follows: for every $c \in f(B)$, $\pi(c)$ is defined to be the unique $\alpha$ so that $c \in f(B_\alpha)$, i.e. the slope of any point $b \in B$ inducing $c$. Since each $f(B_\alpha)$ is shift-invariant, clearly $\pi(\sigma(c)) = \pi(c) = \textrm{id}(\pi(c))$ for every $c \in f(B)$. Also, $\pi$ is clearly surjective. It remains only to show that $\pi$ is continuous. 

Consider any sequence $(c_n) \in f(B)$ for which $c_n \rightarrow c$. Define $\alpha_n = \pi(c_n)$ and $\alpha = \pi(c)$, so that $c_n \in f(B_{\alpha_n})$ and $c \in f(B_\alpha)$. It remains to prove that $\alpha_n \rightarrow \alpha$. Since $c_n \in f(B_{\alpha_n})$, there exists $b_n \in B_{\alpha_n}$ inducing $c_n$, and similarly there exists $b \in B_\alpha$ inducing $c$. Clearly, since $c_n \rightarrow c$, it must be the case that $b_n \rightarrow b$ as well.

By Corollary~\ref{within1}, the slope of any balanced sequence containing the word $b([1,k])$ is trapped between 
$\frac{\#(b([1,k]),1) - 1}{k}$ and $\frac{\#(b([1,k]),1) + 1}{k}$. Since for every $k$, $b_n([1,k])$ eventually agrees with $b([1,k])$, 
$\alpha_n \rightarrow \alpha$, completing the proof that $B$ does not have TCPE.\\

\end{proof}

\noindent
\textbf{Claim B2:} $(f(B), \sigma)$ has ZTCPE.

\begin{proof}
Consider any surjective factor map $\psi: (f(B), \sigma) \rightarrow (Y, S)$ where $h(Y,S) = 0$ and $Y$ is a zero-dimensional topological space. We must show that $|Y| = 1$.
We first note that all Sturmian shifts are minimal (see Chapter 2 of \cite{lothaire}) and so satisfy (3a) and (3b) in Theorem~\ref{blackbox}. Therefore, for every irrational $\alpha$, $f(B_\alpha)$ has TCPE by Theorem~\ref{blackbox}.
For every $\alpha$, since $\psi(f(B_\alpha)) \subset Y$, clearly $(\psi(f(B_\alpha)), S)$ has zero entropy as well.
Therefore, for every $\alpha \notin \mathbb{Q}$, $\psi(f(B_{\alpha}))$ consists of a single point, call it $g(\alpha)$.

Now we consider the more complicated case of rational $\alpha$. We first define $B_{\alpha,0} \subset B_{\alpha}$ to be the orbit $O(\underline{x_{\alpha}})$ of the lower
characteristic sequence for $\alpha$ defined in Section~\ref{defs}. Then $B_{\alpha,0}$ is a single periodic orbit and so satisfies (3a) and (3b) in Theorem~\ref{blackbox}, therefore $f(B_{\alpha,0})$ has TCPE, and as above, $\psi(f(B_{\alpha,0}))$ consists of a single point, which we again denote by $g(\alpha)$. 

We have now defined $g$ on all of $[0,1]$, and claim that it is continuous. First, by Lemma~\ref{charconv}, for any sequence $\alpha_n \in [0,1]$ converging to a limit $\alpha$ from above, the corresponding
lower characteristic sequences $\underline{x_{\alpha_n}} \in B_{\alpha_n}$ converge to $\underline{x_{\alpha}} \in B_{\alpha,0}$. Then, we can define $\underline{y_{\alpha_n}} \in f(B_{\alpha_n})$ induced by 
$\underline{x_{\alpha_n}}$ and $\underline{y_{\alpha}} \in f(B_{\alpha})$ induced by $\underline{x_{\alpha}}$ so that $\underline{y_{\alpha_n}} \rightarrow \underline{y_{\alpha}}$; just give them all the same pattern of $0$ and non-$0$ symbols (say by making all gaps of $0$s have length $3$). Continuity of $\psi$ then means that $\psi(\underline{y_{\alpha_n}}) = g(\alpha_n)$ approaches $\psi(\underline{y_{\alpha}}) = g(\alpha)$, proving that $g$ is continuous from the right. A similar argument using upper characteristic sequences proves that $g$ is also continuous from the left, and therefore continuous.

Now, choose any $\alpha \in \mathbb{Q}$ and $c_{\alpha} \in f(B_\alpha)$. Then $c_{\alpha}$ is induced by some $b_{\alpha} \in B_\alpha$, and by Lemma~\ref{balsturm}, $b_{\alpha}$ is the limit of a sequence $b_{\alpha_n} \in B_{\alpha_n}$ for some sequence of irrational $\alpha_n$ converging to $\alpha$. Then, we can define $c_{\alpha_n} \in f(B_{\alpha_n})$ so that $c_{\alpha_n} \rightarrow c_{\alpha}$ by using the same structure of $0$ and non-$0$ symbols as $c_{\alpha}$ for all $c_{\alpha_n}$. Then by continuity, $\psi(c_{\alpha_n}) = g(\alpha_n)$ converges to $\psi(c_{\alpha})$, implying that $\psi(c_{\alpha}) = g(\alpha)$ by continuity of $g$. 
We've then shown that $\psi$ collapses every $f(B_\alpha)$ to a single point $g(\alpha)$ for a continuous function $g$ on $[0,1]$. 
Then $g([0,1]) = Y$ must be connected (as the continuous image of a connected set), and the only connected subsets of $Y$ are singletons. 
We have therefore shown that $g$ is constant, and so $|Y| = 1$. Since $\psi$ was arbitrary, this shows that $(f(B), \sigma)$ has ZTCPE.

\end{proof}

We've shown that $(f(B), \sigma)$ has ZTCPE but not TCPE, completing the proof of Theorem~\ref{1dsub}. 

\end{proof}

We are finally ready to present the proof of Theorem~\ref{2dSFT}. It is quite similar to that of Theorem~\ref{1dsub}, but requires a somewhat technical description of a $\mathbb{Z}^2$-SFT which will play the role of $B$ from the former proof.

\begin{proof}[Proof of Theorem~\ref{2dSFT}]

We begin with the description of a $\mathbb{Z}^2$-SFT $X$ in which all rows of points in $X$ have a property similar to being balanced. The alphabet is $A = \{0,1\}^3$, and the rules are as follows:\\

\noindent
$\bullet$ The first coordinate is constant in the vertical direction, i.e. for any $x \in X$ and $(i,j) \in \mathbb{Z}^2$, $(x(i,j))(1) = (x(i,j+1))(1)$.

\noindent
$\bullet$ The second coordinate is constant along the line $y = x$, i.e. for any $x \in X$ and $(i,j) \in \mathbb{Z}^2$, $(x(i,j))(2) = (x(i+1,j+1))(2)$.

\noindent
$\bullet$ The third coordinate is the difference between the running totals of the first two coordinates, i.e. for any $x \in X$ and $(i,j) \in \mathbb{Z}^2$, $(x(i,j))(3) = (x(i-1,j))(3) + (x(i,j))(2) - (x(i,j))(1)$.\\

It should be clear that $X$ is an SFT. For any $x \in X$, define $a(x), b(x) \in \{0,1\}^{\mathbb{Z}}$ by $(a(x))(n) = (x(n,0))(1)$ and $(b(x))(n) = (x(n,0))(2)$. Note that $a(x)$ and $b(x)$ completely determine the first and second coordinates of $x$ due to the constancy of the first and second coordinates in their respective directions. Also note that given the first and second coordinates in a row of a point of $x$,
the third coordinate along that row is determined up to an additive constant. For any row of a point of $x$ in which the first and second coordinates along a row are not equal sequences, 
the third coordinate contains a $0$ and $1$ and therefore is completely forced by the first two coordinates (since no constant can be added to keep the third coordinate using only $0$ and $1$). 
This means that $a(x)$ and $b(x)$ uniquely determine $x$ as long as $a(x)$ and $b(x)$ are not shifts of each other. If $a(x) = \sigma^n(b(x))$ for some $n$, then the $n$th row of $x$ has first and second coordinates both equal to $a(x)$,
meaning that the third coordinate may either be all $0$s or all $1$s along that row. Similar facts are true even for finite patterns; in any rectangular pattern, the first and second coordinates along a row force the third unless the first and second coordinates are equal words along that row, in which case it is locally allowed for the third coordinate to either be all $0$s or all $1$s. We note that this does not necessarily mean that both choices are globally admissible; it may be the case that a rectangular pattern with equal first and second coordinates along a row can only be extended in such a way that the first and second coordinates along that row are eventually unequal, forcing the entire row.

Since $a(x)$ and $b(x)$ determine $x$ up to some possible constant third coordinates of rows, we wish to understand the structure of which pairs $a(x), b(x)$ may appear for $x \in X$, for which we need a definition.

\begin{definition}\label{jbdef}
Two sequences $a,b \in \{0,1\}^{\mathbb{Z}}$ are \textbf{jointly balanced} if for every $n$ and every pair of subwords $w,w'$ of $a,b$ of length $n$, the numbers of $1$s in $w$ and $w'$ differ by at most $1$, i.e. $|\#(w,1) - \#(w',1)| \leq 1$.
\end{definition}

\noindent
\textbf{Claim C1:} There exists $x \in X$ with $a(x) = a$ and $b(x) = b$ if and only if $a$ and $b$ are jointly balanced.

\begin{proof}

\noindent
$\Longrightarrow$: Consider any $x \in X$ and arbitrary $n$-letter subwords $w = (a(x))([i,i+n-1])$ of $a(x)$ and 
$w' = (b(x))([j, j + n - 1])$ of $b(x)$. The $(i-j)$th row of $x$ contains $a(x)$ and 
$\sigma^{i-j} b(x)$ as its first two coordinates, and by the third rule defining $X$, 
\begin{multline*}
(x(i+n-1,i-j))(3) - (x(i-1,i-j))(3) = \\
\sum_{k = i}^{i + n - 1} (x(k,i-j))(2) - (x(k,i-j))(1) = \\
\sum_{k = i}^{i + n - 1} (x(k,i-j))(2) - \sum_{k = i}^{i + n - 1} (x(k,i-j))(1) = \\
\sum_{\ell = j}^{j + n - 1} (b(x))(\ell) - \sum_{k = i}^{i + n - 1} (a(x))(k) = 
\#(w',1) - \#(w,1).
\end{multline*}

Since $x(i+n-1,i-j))(3)$ and $(x(i-1,i-j))(3)$ are either $0$ or $1$, their difference is $-1$, $0$, or $1$, and so $a(x)$ and $b(x)$ are jointly balanced.\\

\noindent
$\Longleftarrow$: Suppose that $a$ and $b$ are jointly balanced. Then, define $x \in (\{0,1\}^2 \times \{-1,0,1\})^{\mathbb{Z}^2}$ as follows. The first two coordinates are given by
$(x(i,j))(1) = a(i)$ and $(x(i,j))(2) = b(j - i)$ for every $(i,j) \in \mathbb{Z}^2$. The third coordinate is defined piecewise. Firstly, $(x(0,j))(3) = 0$ for all $j \in \mathbb{Z}$. 
For $i > 0$, $(x(i,j))(3) = \#(b([1-j,i-j]),1) - \#(a([1,i]),1)$. Finally, for $i < 0$, $(x(i,j))(3) = -\#(b([i-j,-j]),1) + \#(a([i,0]),1)$. By joint balancedness,
the third coordinate clearly takes only the values $-1$, $0$, and $-1$. The reader may check that $x$ satisfies the rules (the three from the bulleted list given at the beginning of the proof) defining $X$, however it may not be a point of $X$ since its third coordinate may take the value $-1$. However, it is not possible for the third coordinate of any row of $x$ to contain $1$ and $-1$; if $(x(i,j))(3) = 1$, $(x(i',j))(3) = -1$, and $i > i'$, then $(x(i,j))(3) - (x(i',j))(3) = \#(b([i'-j,i-j]),1) - \#(a([i',i]),1) = 2$, a contradiction to joint balancedness of $a$ and $b$. (The case $i < i'$ is trivially similar.)
Therefore, for any $j$ at which the $j$th row of $x$ contains $-1$, that row can only contain $-1$ and $0$, and so we simply add $1$ to the third coordinate of that entire row. 
This new point, call it $x' \in (\{0,1\}^3)^{\mathbb{Z}^2}$, is a point of $X$ with $a(x') = a$ and $b(x') = b$, completing the proof.

\end{proof}

We will now classify the jointly balanced pairs $(a,b)$.\\

\noindent
\textbf{Claim C2:} All jointly balanced pairs $(a,b)$ fall into at least one of the following four categories:\\

\noindent
(1) $\alpha \notin \mathbb{Q}$ and $a$ and $b$ are $1$-balanced sequences with slope $\alpha$

\noindent
(2) $\alpha \in \mathbb{Q}$ and $a$ and $b$ are $1$-balanced sequences with slope $\alpha$

\noindent
(3) $a$ is $2$-balanced, jointly balanced with $\underline{x_{\alpha}}$, and $b \in O(\underline{x_{\alpha}})$

\noindent
(4) $b$ is $2$-balanced, jointly balanced with $\underline{x_{\alpha}}$, and $a \in O(\underline{x_{\alpha}})$.

\begin{proof}

Firstly, if $(a,b)$ are jointly balanced, then clearly both $a$ and $b$ are $2$-balanced; any two $n$-letter subwords of $a$ have number of $1$s within $1$ of some $n$-letter subword of $b$, and so their numbers of $1$s may differ by at most $2$. This implies by Lemma~\ref{balfreq} that $a$ and $b$ both have some uniform frequency of $1$s (or slope), which must be the same since $a,b$ are jointly balanced. 

Next, suppose that $a$ is $2$-balanced, but not $1$-balanced. Then, there exist $n$ and two $n$-letter subwords 
$v,v'$ of $a$ with $\#(v,1)$ and $\#(v',1)$ differing by $2$, say that they are $k$ and $k + 2$ respectively. But then, since $a$ and $b$ are jointly balanced, every $n$-letter subword of $b$ must have exactly $k + 1$ $1$s. This implies that $b$ is periodic with period $n$ (this is not necessarily the least period of $b$ though). We claim that $b$ must be $1$-balanced as well. 

Assume for a contradiction that $a,b$ are both $2$-balanced but not $1$-balanced. Then there are $m,n$ (which we take to be minimal), two 
$n$-letter subwords $v,v'$ of $a$ with $|\#(v,1) - \#(v',1)| = 2$, and two $m$-letter subwords $w,w'$ of $b$ with with 
$|\#(w,1) - \#(w',1)| = 2$. Without loss of generality, assume $n \leq m$. As above, every $n$-letter subword of 
$b$ has the same number of $1$s, and so we can remove the first $n$ letters of $w,w'$ to yield shorter words with the same property.
However, this contradicts minimality of $m$. Therefore, our assumption was wrong, and if $a$ is $2$-balanced but not $1$-balanced, then 
$b$ is $1$-balanced. Then, by Lemma~\ref{balper}, since $b$ is periodic, it must be in the orbit $O(\underline{x_{\alpha}})$ of the lower characteristic sequence $\underline{x_{\alpha}}$.

We conclude several things from this: if $(a,b)$ is jointly balanced and $a$ is not $1$-balanced, then $a$ is $2$-balanced, $b$ is periodic, both have the same rational slope $\alpha$, and $b \in O(\underline{x_{\alpha}})$. Similarly, if $b$ is not $1$-balanced, then $b$ is $2$-balanced, $a$ is periodic, both have the same rational slope $\alpha$, and $a \in O(\underline{x_{\alpha}})$. These correspond to categories 
(3) and (4). 

This means that if $a$ and $b$ have irrational slope $\alpha$, then they must both be $1$-balanced; this corresponds to category (1). The only remaining case is that $a$ and $b$ have rational slope $\alpha$ and both are $1$-balanced; this corresponds to category (2) and completes the proof.

\end{proof}

We briefly note that by Corollary~\ref{bal2choice}, any two $1$-balanced sequences with irrational slope $\alpha$ are jointly balanced, and so all $(a,b)$ in category (1) are jointly balanced. By description, all $(a,b)$ in categories (3) and (4) are clearly jointly balanced, but 
$(a,b)$ in category (2) need not be jointly balanced; for instance, $a = \ldots 0101001010 \ldots$ and $b = \ldots 1010110101 \ldots$ are both $1$-balanced sequences with slope $\alpha = \frac{1}{2}$, but $a$ contains $00$ and $b$ contains $11$, and so $a$ and $b$ are not jointly balanced. 

For any $\alpha \in [0,1]$, we write $X_\alpha = \{x \in X \ : \ a(x), b(x) \textrm{ have slope } \alpha\}$; clearly $X = \bigsqcup X_\alpha$ and so by (1) of Theorem~\ref{blackbox}, $f(X) = \bigsqcup f(X_{\alpha})$. By (2) of Theorem~\ref{blackbox}, $f(X)$ is a $\mathbb{Z}^2$-SFT. We will show that $f(X)$ has ZTCPE but not TCPE, for which we need to prove several properties about the subshifts $f(X_{\alpha})$.\\

\noindent
\textbf{Claim C3:} For every $\alpha \notin \mathbb{Q}$, $f(X_{\alpha})$ has TCPE.

\begin{proof}

Choose any $\alpha \notin \mathbb{Q}$. We will show that $X_{\alpha}$ satisfies (3a) and (3b) from Theorem~\ref{blackbox}, which will imply that $f(X_{\alpha})$ has TCPE. 
Choose any pattern $w \in L_{[-n,n]^2}(X_{\alpha})$, and define $x \in X_{\alpha}$ for which $x([-n,n]^2) = w$. Then $a(x)$ and $b(x)$ are balanced sequences with irrational slope $\alpha$, and therefore Sturmian with slope $\alpha$. 
We break into two cases. 

If $a(x) \neq \sigma^i b(x)$ for all $i \in [-n,n]$, then there exists $N$ so that all rows of $w' := x([-N,N] \times [-n,n])$ have unequal first and second coordinates, and therefore the first two coordinates of $w'$ force the third. Then, define the words $u = (a(x))([-N-n,N+n])$ and $v = (b(x))([-N-n,N+n])$; by the rules defining $X$, for $y \in X$, if the first two coordinates of 
$y([-N-n, N+n] \times \{0\})$ are $u$ and $v$, then the first two coordinates of $y([-N,N] \times [-n,n])$ match those of $w' = x([-N,N] \times [-n,n])$, and since the first and second coordinates of $w'$ force the third, 
$y([-N,N] \times [-n,n]) = w'$ and $y([-n,n]^2) = w$. We say that $u$ and $v$ force an occurrence of $w$ in any point of $X$.

If $a(x) = \sigma^i b(x)$ for some $i \in [-n,n]$, then we wish to slightly change one of $a(x)$ and $b(x)$ to another Sturmian sequence with the same slope $\alpha$ so that they are no longer shifts of one another, but without changing $x([-n,n])^2$. Since $a(x)$ and $b(x)$ are not periodic, for all $j \neq i$, $a(x) \neq \sigma^j b(x)$. Then we can extend $w$ to $w' := x([-N, N] \times [-n,n])$ for which every row has unequal first and second coordinates, except for the $i$th row, which must have equal first and second coordinates since 
$a(x) = \sigma^i b(x)$. Recall that $a(x) = \sigma^i b(x)$ is Sturmian, and so can be written as a version of a lower/upper characteristic sequences with a constant added inside the floor/ceiling function. Choose another Sturmian sequence $a'$ with slope $\alpha$ which is not equal to $a(x)$, but with $a'([-N-n, N+n]) = a([-N-n, N+n])$; this can be accomplished by adding a tiny constant inside the floor/ceiling function defining $a(x)$. 
Then define $k$ to be the minimal positive integer for which $(a(x))(k) \neq a'(k)$, and $j$ to be the maximal such negative integer. Then it cannot be the case that $(a(x))(j) = (a(x))(k) = 0$; if so, then $a'(j) = a'(k) = 1$,
and $|\#((a(x))([j,k]), 1) - \#(a'([j,k]), 1)| = 2$, contradicting the fact that any two Sturmian sequences with slope $\alpha$ are jointly balanced. Similarly, $(a(x))(j) = (a(x))(k) = 1$ is impossible. Therefore,
either $(a(x))(j) = 0$, $(a(x))(k) = 1$, $a'(j) = 1$, and $a'(k) = 0$, or all of these values are the opposite. We assume the former, as the proof of the latter is almost exactly the same.
We use Claim C1 to define $x' \in X_{\alpha}$ with $a(x') = a'$ and $b(x') = b(x)$. Then since $a'([-N-n, N+n]) = a([-N-n, N+n])$, the first two coordinates of $x'([-N,N] \times [-n,n])$ and $w' = x([-N,N] \times [-n,n])$ are equal, and therefore they have the same third coordinates as well, except possibly in the $i$th row. In the $i$th row of $x'$, the first and second coordinates are $a(x') = a'$ and $\sigma^i b(x') = \sigma^i b(x) = a(x)$ respectively. Then $x'(i,j)$ has first and second coordinates $a'(j) = 1$ and $(a(x))(j) = 0$ respectively, implying that the third coordinate of $x'(i,j)$ is $0$. Since $a'$ and $a(x)$ agree on $(j,k)$, the third coordinate of the $i$th row of
$x'$ is $0$ on that entire interval, including $[-N,N]$. If the third coordinate of $w'$ is $0$ on the entire $i$th row, then we have constructed $x'$ with $x'([-N,N] \times [n,n]) = w'$, and so $x'([-n,n]^2) = w$.
If the third coordinate of $w'$ was instead $1$ on the $i$th row, then the reader may check that if we define $x'$ instead by $a(x') = a(x)$ and $b(x') = \sigma^{-i} a'$, then $x'$ would have third
coordinate $1$ on the $i$th row and again $x'([-n,n]^2) = w$. In either case, $x'([-n,n]^2) = w$ and $a(x') \neq \sigma^i b(x')$ for all $i \in [-n,n]$, and so the argument of the last paragraph yields
a finite pair of words $u$ and $v$ which force an occurrence of $w$ in any point of $X$.

Now, since Sturmian subshifts satisfy (3a) from Theorem~\ref{blackbox}, there exist points $a_u$ and $b_v$ which contain occurrences of $u$ and $v$ respectively, beginning at sets of indices $A,B$ with positive density in $\mathbb{Z}$ (in fact every point has this property). By Claim C1, define $x''$ with $a(x'') = a_u$ and $b(x'') = b_v$. Then, for every pair $(i,j)$ with $i \in A$ and $i + j \in B$, $x''(i,j)$ begins occurrences of $u$ and $v$ on its first two coordinates, yielding an occurrence of $w$. The set of such $(i,j)$ has positive density in $\mathbb{Z}^2$ since $A,B$ had positive density in 
$\mathbb{Z}$, and so we have proved that $X_{\alpha}$ satisfies (3a) from the hypotheses of Theorem~\ref{blackbox}.

Now, consider any two patterns $w,w' \in L_{[-n,n]^2}(X_{\alpha})$. As above, we can find words $u,v,u',v'$ on $\{0,1\}^2$ so that a location containing $u,v$ on its first two coordinates forces an occurrence of $w$, and $u',v'$ 
similarly force $w'$. Since Sturmian subshifts are minimal, there exists Sturmian $a$ with slope $\alpha$ containing both $u$ and $u'$; say $a(i)$ begins an occurrence of $u$ and $a(j)$ begins an occurrence of $u'$. Similarly, there exists Sturmian $b$ with slope $\alpha$ containing both $v$ and $v'$; say $b(k)$ begins an occurrence of $v$ and $b(\ell)$ begins an occurrence of $v'$. 
By Claim C1, we may define $x \in X_{\alpha}$ with $a(x) = a$ and $b(x) = b$. Then, $x(i,i-k)$ begins occurrences of $u$ and $v$ in the first two coordinates (forcing an occurrence of $w$), and
$x(j, j - \ell)$ begins occurrences of $u'$ and $v'$ in the first two coordinates (forcing an occurrence of $w'$). Therefore, $x$ contains both $w$ and $w'$, verifying (3b) from the hypotheses of Theorem~\ref{blackbox}. We then know that $f(X_{\alpha})$ has TCPE.

\end{proof}

We now move to the more difficult case of rational $\alpha$. By Claim C2, we know that $X_{\alpha}$ can be written as the union of three subshifts, defined as follows:\\

\noindent
$\bullet$ $X_{\alpha,1}$ consists of $x \in X_{\alpha}$ for which $a(x)$ and $b(x)$ are both $1$-balanced with slope $\alpha$

\noindent
$\bullet$ $X_{\alpha,2}$ consists of $x \in X_{\alpha}$ for which $a(x)$ is $2$-balanced and jointly balanced with $\underline{x_{\alpha}}$, and $b(x) \in O(\underline{x_{\alpha}})$

\noindent
$\bullet$ $X_{\alpha,3}$ consists of $x \in X_{\alpha}$ for which $b(x)$ is $2$-balanced and jointly balanced with $\underline{x_{\alpha}}$, and $a(x) \in O(\underline{x_{\alpha}})$.\\

We also define a useful subshift of $X_{\alpha,1}$:\\

\noindent
$\bullet$ $X_{\alpha,0} \subset X_{\alpha,1}$ consists of $x \in X_{\alpha}$ for which $a(x), b(x) \in O(\underline{x_{\alpha}})$.\\

\noindent
Every point $x \in X_{\alpha,0}$ has a strange property: since $a(x)$ and $b(x)$ are in the same periodic orbit, there are infinitely many rows for which the first and second coordinates are the same,
and so in each of those rows the third coordinate can be chosen to be all $0$s or all $1$s, each independently of every other such row. This in fact means that many points in $X_{\alpha,0}$ are in fact not limits of points of $X_{\alpha_n}$ for irrational $\alpha_n$, in contrast to the proof of Theorem~\ref{1dsub}. We instead use the following fact.\\

\noindent
\textbf{Claim C4:} For every $\alpha \in \mathbb{Q}$, $f(X_{\alpha,0})$ has TCPE.

\begin{proof}
Choose any $\alpha \in \mathbb{Q}$, and any pattern $w$ in $L_{[-n,n]^2}(X_{\alpha,0})$. For each row, we extend $w$ on the left and
right to make the first and second coordinates different if possible, arriving at a new pattern $w' \in L_{[-N,N] \times [-n,n]}(X_{\alpha_0})$ with the following property: for each row where the first and second coordinates of $w'$ match, any point $x \in X_{\alpha,0}$ with 
$x([-N,N] \times [-n,n]) = w'$ must have equal first and second coordinates in that entire row. Define by $w''$ the pattern given by the first
two coordinates of $w'$. Choose an arbitrary point $x \in X_{\alpha,0}$; $w''$ clearly appears with positive frequency (in fact along a subgroup of finite index of $\mathbb{Z}^2$) in $x$ since the first and second coordinates of $x$ come from the single periodic orbit 
$O(\underline{x_{\alpha}})$. At each occurrence of $w''$ in $x$, the third coordinate is forced to match that of $w'$ except possibly in some rows where the first and second coordinates are equal. Recall though that this forces that entire row of $x$ to have equal first and second coordinates, and so the third coordinates of any such rows in $x$ can be judiciously changed to create a point $x' \in X_{\alpha,0}$ in which $w'$ itself (and therefore $w$ as well) appears with positive frequency, proving (3a) from Theorem~\ref{blackbox}.

Now choose any $v,w \in L_{[-n,n]^2}(X_{\alpha,0})$. As above, $v$ and $w$ can be extended to
$v',w' \in L_{[-N,N] \times [-n,n]}(X_{\alpha,0})$ such that any rows with equal first and second coordinates within $v'$ (or $w'$) force equal first and second coordinates throughout the corresponding entire biinfinite row of any point of $X_{\alpha,0}$ containing $v'$ (or $w'$).
Then, if we denote by $v''$ and $w''$ the patterns given by the first two coordinates of $v'$ and $w'$ respectively, and choose any 
$x \in X_{\alpha,0}$, then again $v''$ and $w''$ appear with positive frequency in $x$. Therefore, it's possible to choose occurrences of $v''$ and $w''$ within $x$ which share no row. Then as above, for any rows in which $v''$ and $w''$ have equal first and second coordinates, the corresponding rows of $x$ have equal first and second coordinates. The third coordinate on those rows can then be changed (if necessary) to create a new point $x'$ in which $v'$ and $w'$ (and therefore $v$ and $w$) both appear, verifying (3b) and implying that
$f(X_{\alpha,0})$ has TCPE via Theorem~\ref{blackbox}. 

\end{proof}

Now, similarly to the proof of Theorem~\ref{1dsub}, we wish to deal with points of $f(X_{\alpha,1}) \setminus f(X_{\alpha,0})$ by representing them as limits of points from the simpler irrational case.\\

\noindent
\textbf{Claim C5:} Every $c \in f(X_{\alpha,1}) \setminus f(X_{\alpha,0})$ can be written as the limit of a sequence $c_n \in f(X_{\alpha_n})$ for some sequence of irrational $\alpha_n$ converging to $\alpha$.

\begin{proof}
Choose any $c \in f(X_{\alpha,1}) \setminus f(X_{\alpha,0})$, which is induced by some $x \in X_{\alpha,1} \setminus X_{\alpha,0}$. We will show that $x$ can be written as a limit of $x_n \in X_{\alpha_n}$ as claimed; then clearly we can create $c_n \in f(X_{\alpha_n})$ converging to $c$ by simply copying the ``ribbon structure'' of $c$.

Since $x \in X_{\alpha,1} \setminus X_{\alpha,0}$, at least one of $a(x)$ and $b(x)$ is not periodic (though in fact both must be eventually periodic). We break into two cases depending on whether or not $a(x)$ and $b(x)$ are shifts of each other.\\

\textbf{Case 1.} Suppose that $a(x)$ and $b(x)$ are not shifts of each other. Then for every $n$, there exists $N$ so that
all rows of $x([-N,N] \times [-n,n])$ have unequal first and second coordinates, and thereby the third coordinate 
of $x([-N,N] \times [-n,n])$ is forced by the first two coordinates on that pattern. 
Write $u_n := (a(x))([-N-n, N+n])$ and $v_n := (b(x))([-N-n, N+n])$. Since the first and second
coordinates of all rows of $x$ are balanced sequences, $u_n$ and $v_n$ are balanced words and
so by Lemma~\ref{balsturm} there exist Sturmian sequences 
$a_n$ and $b_n$ for which $a_n([-N-n, N+n]) = u_n$ and $b_n([-N-n,N+n]) = v_n$. 

We in fact wish to choose $a_n$ and $b_n$ with the same slope, which requires a more detailed examination of the proof 
of Lemma~\ref{balsturm} from \cite{lothaire}. In that proof, it is shown that in fact a balanced word $w$ is
a subword of any Sturmian sequence with irrational slope strictly between 
\[
\alpha'(w) := \max_v \left(\frac{\#(v,1) - 1}{|v|}\right) \textrm{ and } \alpha''(w) := \min_v \left(\frac{\#(v,1) + 1}{|v|}\right),
\]
where $v$ ranges over all subwords of $w$. We then need to show that $(\alpha'(u_n), \alpha''(u_n)) \cap (\alpha'(v_n), \alpha''(v_n)) \neq \varnothing$.
First, note that since $u_n$ and $v_n$ are subwords of balanced sequences with slope $\alpha$, by Corollary~\ref{within1}
$\alpha'(u_n), \alpha'(v_n) \leq \alpha \leq \alpha''(u_n), \alpha''(v_n)$. The only case in which we are not done is if either
$\alpha'(u_n) = \alpha = \alpha''(v_n)$ or $\alpha'(v_n) = \alpha = \alpha''(u_n)$. For a contradiction, we assume the former;
the other case is trivially similar. Since $\alpha'(u_n) = \alpha$, $u_n$ has a subword $s$ with $\#(s,1) = |s| \alpha + 1$. 
If we write $\alpha = \frac{i}{j}$ in lowest
terms, then $|s|$ must be a multiple of $j$ since $\#(s,1)$ is an integer. But then we may partition $s$ into $j$-letter subwords, and
one of them, call it $s'$, must have $\#(s',1) = i + 1$. However, a similar argument shows that since $\alpha''(v_n) = \alpha$,
$v_n$ contains a $j$-letter subword $t'$ with $\#(t',1) = i - 1$, which violates the fact that $u_n$ and $v_n$ are jointly balanced.
Therefore, $(\alpha'(u_n), \alpha''(u_n)) \cap (\alpha'(v_n), \alpha''(v_n)) \neq \varnothing$, and so we may choose
$a_n$ and $b_n$ to have the same irrational slope $\alpha_n$.

Then $a_n$ and $b_n$ are jointly balanced, so by Claim C1 we may define $x_n \in X_{\alpha_n}$ with $a(x_n) = a_n$ and $b(x_n) = b_n$.
Then $x_n([-N,N] \times [-n,n])$ has first two coordinates agreeing with those of $x([-N,N] \times [-n,n])$, and we argued above
that their third coordinates must agree as well, meaning that $x_n([-N,N] \times [-n,n]) = x([-N,N] \times [-n,n])$. 
Therefore, $x$ is the limit of the sequence $x_n \in X_{\alpha_n}$ for a sequence of irrational $\alpha_n$ converging to $\alpha$.\\ 

\textbf{Case 2.} Suppose that $a(x)$ and $b(x)$ are shifts of each other, say $a(x) = \sigma^k b(x)$. Then both $a(x)$ and $b(x)$
are not periodic (since $x \notin X_{\alpha,0}$), and so for $m \neq k$, $a(x) \neq \sigma^m b(x)$. Therefore, for every $n > |k|$, we can choose $N$ 
so that all rows of $x([-N,N] \times [-n,n])$ except the $k$th have unequal first and second coordinates, and thereby the third coordinate 
of $x([-n,n] \times [-N,N])$ is forced by the first two coordinates on each row except the $k$th. Then, again by Lemma~\ref{balsturm}, there exists a Sturmian sequence $b_n$, with irrational slope $\alpha_n$, for which $b_n([-N-n, N+n]) = (b(x))([-N-n, N+n])$. 
By Claim C1, we define $x_n \in X_{\alpha_n}$ with $a(x_n) = \sigma^k b_n$ and $b(x_n) = b_n$. Then $x_n([-N,N] \times [-n,n])$ has first two coordinates agreeing with those of $x([-N,N] \times [-n,n])$, and by the above argument, the third coordinates are forced to agree as well, except possibly on the $k$th row. However, the $k$th row of $x_n$ has equal first and second coordinates, and so the third coordinate can be chosen to be either all $0$s or all $1$s, whichever matches the third coordinate of the $k$th row of $x([-N,N] \times [-n,n])$.
Then $x_n([-N,N] \times [-n,n]) = x([-N,N] \times [-n,n])$, and so again $x$ is the limit of the sequence $x_n \in X_{\alpha_n}$ for a sequence of irrational $\alpha_n$ converging to $\alpha$.\\

Then, the sequence $x_n$ induces a sequence $c_n \in f(X_{\alpha_n})$ with the same ``ribbon structure'' as that of $c$, and clearly
$c_n \rightarrow c$, completing the proof.

\end{proof}

Finally we must treat the subshifts $X_{\alpha,2}$ and $X_{\alpha,3}$. We first treat the special cases $\alpha = 0$ and $\alpha = 1$.\\

\noindent
\textbf{Claim C6:} For $\alpha \in \{0,1\}$, $X_{\alpha,2} \cup X_{\alpha,3} \subseteq X_{\alpha,1}$.

\begin{proof}

We treat only $\alpha = 0$, as $\alpha = 1$ is trivially similar. Note that $\underline{x_{\alpha}} = 0^{\infty} = \ldots 000 \ldots$, and the only sequences jointly balanced with $\underline{x_{\alpha}}$ are $\underline{x_{\alpha}}$ itself and the orbit of $0^{\infty}1 0^{\infty} = \ldots 0001000 \ldots$. All of these sequences are, however, also $1$-balanced. Therefore, $X_{0,2} \cup X_{0,3} \subseteq X_{0,1}$, and similarly $X_{1,2} \cup X_{1,3} \subseteq X_{1,1}$. 

\end{proof}

A key technique used in the proof of Theorem~\ref{1dsub} was to show that any point of $B_{\alpha}$ could be written as the limit of a sequence of points from $B_{\alpha_n}$ for some irrational $\alpha_n \rightarrow \alpha$. However, the analogous fact here is not true; the set of $1$-balanced sequences is closed, so no point in $X_{\alpha,2} \cup X_{\alpha,3}$ (where one of $a(x)$ or $b(x)$ is not $1$-balanced)
can be written as the limit of a sequence from $X_{\alpha_n}$ for irrational $\alpha_n$ (for which both $a(x)$ and $b(x)$ are $1$-balanced). Instead, we will prove the following.\\

\noindent
\textbf{Claim C7:} For $\alpha \in \mathbb{Q} \setminus \{0,1\}$, $f(X_{\alpha,2} \cup X_{\alpha,3})$ has TCPE.

\begin{proof}
We need only show that $X_{\alpha,2} \cup X_{\alpha,3}$ satisfies (3a) and (3b) from the hypotheses of Theorem~\ref{blackbox}. 
For the first part, consider any $w \in L_{[-n,n]^2}(X_{\alpha,2})$. As in the proof of Claim C4, we may extend $w$ to 
$w' \in L_{[-N,N] \times [-n,n]}(X_{\alpha,2})$ with the following property: for each row of $w'$ with equal
first and second coordinates, the first and second coordinates of the entire corresponding row are forced to agree
for any $x \in X_{\alpha,2}$ with $x([-N,N] \times [-n,n]) = w'$. Choose such an $x \in X_{\alpha,2}$. 
By definition, $b(x) \in O(\underline{x_\alpha})$ and $a(x)$ is jointly balanced with $\underline{x_\alpha}$.
Define $u = (a(x))([-N-n, N+n])$ and $v = (b(x))([-N-n,N+n])$; by the rules defining $X$, any $x' \in X_{\alpha,2}$ with $(a(x'))([-N-n,N+n]) = u$ and $(b(x'))([-N-n,N+n]) = v$
must have first and second coordinates on $[-N,N] \times [-n,n]$ matching those of $w'$. We now wish to show that there exists
$x' \in X_{\alpha,2}$ for which the pair $a(x'), b(x')$ sees the pair $u,v$ with positive frequency. 
We begin by proving that we can find a periodic sequence $a$ which is jointly balanced with $\underline{x_{\alpha}}$ and for which 
$a(-N-n,N+n) = u$. Recall that $u$ is contained in $a(x)$, which is jointly balanced with $\underline{x_{\alpha}}$. 
Let's write $\alpha = \frac{i}{j}$ in lowest terms; then by definition, $\underline{x_{\alpha}}$ is periodic with period $j$ 
and every $j$-letter subword of $\underline{x_{\alpha}}$ contains exactly $i$ $1$s. 

We break into two cases. First, assume that there exists some subword $t$ of $a(x)$ which contains $u$, has length $mj$, and contains exactly $mi$ $1$s.
We then claim that $t^{\infty}$ is jointly balanced with $\underline{x_{\alpha}}$. To see this, consider any subword $y$ of $t^{\infty}$; clearly $y$ can be written
as $s t^k p$ for some $p$ a prefix of $t$ and $s$ a suffix of $t$. If $|p| + |s| \leq |t|$, then $t$ can be written as $szp$, and then 
$\#(y,1) = \#(s t^k p, 1) = (k+1)mi - \#(z,1)$. But $\big| \#(z,1) - |z| \alpha \big| \leq 1$ by Corollary~\ref{within1}, so
$\#(y,1)$ is within $1$ of $(k+1)mi - z\alpha = ((k+1)mj - |z|) \alpha = |y| \alpha$. If instead $|p| + |s| > |t|$, then 
we can write $s = z s'$ and $t = ps'$. Then $\#(y,1) = \#(s t^k p,1) = (k+1)mi + \#(z,1)$. Again $\big| \#(z,1) - |z| \alpha \big| \leq 1$,
so $\#(y,1)$ is within $1$ of $(k+1)mi + z\alpha = ((k+1) mj + |z|) \alpha = |y| \alpha$. Either way, we have shown that every subword of
$t^{\infty}$ has number of $1$s within $1$ of $\alpha$ times its length, and so $a := t^{\infty}$ is jointly balanced with $\underline{x_{\alpha}}$. 
Also, since $a(-N-n,N+n)$ was unchanged from $a(x)$, it is equal to $u$.

The remaining case is that every subword $t$ of $a(x)$ with length a multiple of $j$ (say $mj$) which contains $u$ does not have $mi$ $1$s. By the fact that $a(x)$ is jointly balanced with
$\underline{x_{\alpha}}$, the only possibilities are that such subwords have either $mi - 1$ or $mi + 1$ $1$s. If both numbers occurred, then there would have to be an intermediate
subword with length $mj$ containing exactly $mi$ $1$s, a contradiction. 
Therefore, either $\#(t,1) = mi - 1$ for every $mj$-letter subword $t$ of $a(x)$ containing $u$, or $\#(t,1) = mi + 1$ for all such $t$; we treat
only the former case, as the latter is trivially similar. Consider an $mj$-letter subword $t$ of $a(x)$ ending with a $0$ (such a word must exist since $\alpha \neq 1$).
Then $t$ is jointly balanced with $\underline{x_{\alpha}}$, and we claim that if we change the final letter of $t$ to a $1$, yielding a new word $t'$, then $t'$ is jointly balanced
with $\underline{x_{\alpha}}$ as well. To see this, we need only show that every subword of $t'$ has number of $1$s within $1$ of $\alpha$ times its length. Since we changed
only the last letter of $t$, it suffices to show this for suffixes of $t'$. For this purpose, choose any suffix $s'$ of $t'$, and denote by $s$ the
suffix of $t$ of the same length. Take $y$ to be any subword of $a(x)$ ending with $t$ with length a multiple of $j$, say $mj$. 
Then by assumption, $\#(y,1) = mi - 1$. We write $y = zs$, and by Corollary~\ref{within1}, $\#(z,1) \geq |z| \alpha - 1$. 
So, $\#(s,1) \leq (mi - 1) - (|z| \alpha - 1) = (mj - |z|) \alpha = |s| \alpha$. 
Then, again using Corollary~\ref{within1}, $\#(s,1) \in [|s| \alpha - 1, |s| \alpha]$, implying that $\#(s',1) \in [|s'| \alpha, |s'| \alpha + 1]$, since it
is exactly one greater. But then we've shown that $t'$ is jointly balanced with $\underline{x_{\alpha}}$, and it is a word with length $mj$ which contains 
$u$ and has exactly $mi$ $1$s, and so by the previous paragraph, $a:= (t')^{\infty}$ is jointly balanced with $\underline{x_{\alpha}}$ and has $a(-N-n,N+n) = u$.

In both cases, we have found $a$ which is periodic, jointly balanced with $\underline{x_{\alpha}}$, with $a(-N-n,N+n) = u$. 
By Claim C1, define $x' \in X_{\alpha,2}$ with $a(x') = a$ and $b(x') = b(x)$. Then $b(x')$ is periodic and $(b(x')([-N-n,N+n]) = v$, meaning that the pair $u,v$ appears along $a(x'), b(x')$ periodically 
with period the product of those of $a(x'),b(x')$. 
Since $b(x')$ is periodic with period $j$ and the first coordinate of $x'$ is constant vertically, every $j$th row of $x'$ in fact also
contains $u,v$ in its first two coordinates with positive frequency. As explained above, each occurrence of $u,v$ forces a pattern
with shape $[-N,N] \times [-n,n]$ which has the same first two coordinates as $w'$, and for all rows where those coordinates
are unequal, the third coordinate is forced and must match that of $w'$ as well. If any rows have equal first and second coordinates,
then as argued above, the entire associated biinfinite rows of $x'$ must have equal first and second coordinates as well, and then the
third coordinate can be changed (if necessary) in each row to match that of $w'$ in the relevant row. This yields a point
$x'' \in X$ which contains a positive frequency of occurrences of $w'$, and therefore $w$. The proof for patterns in $L(X_{\alpha,3})$ is trivially similar, and so
we have shown (3a) from Theorem~\ref{blackbox} for $X_{\alpha,2} \cup X_{\alpha,3}$.\\

Now we must prove (3b). Again, choose any $w \in L_{[-n,n]^2}(X_{\alpha,2})$, extend to $w' \in L_{[-N,N] \times [-n,n]}(X_{\alpha,2})$ as above, choose $x \in X_{\alpha,2}$ containing $w'$, and define $u$ and $v$ subwords of $a(x)$ and $b(x)$ with the same properties as above. Consider the sequence $a(x)$. It must contain a subword of length $j$ with exactly $i$ $1$s somewhere to the right of $u$; if not, then
as above, every such word would have to have exactly $i-1$ $1$s or every such word would have exactly $i+1$ $1$s, each of which contradicts the fact that $a(x)$ has slope $\alpha = \frac{i}{j}$. Similar reasoning shows that $a(x)$ also contains a subword of length $j$ with exactly $i$ $1$s somewhere to the left of $u$. We then can write the subword of $a(x)$ between these two $j$-letter words (inclusive) as
$p tuv q$, where $p$ and $q$ are length $j$ and $\#(p,1) = \#(q,1) = i$. Since $ptuvq$ was a subword of $a(x)$, it is jointly balanced with 
$\underline{x_{\alpha}}$. We now claim that $a := p^{\infty} tuv q^{\infty}$ is also jointly balanced with $\underline{x_{\alpha}}$. To see this, choose any subword 
$s$ of $a$. We need to show that $\big| \#(s,1) - |s| \alpha \big| \leq 1$. We can clearly write $s$ as
$s = p^k z q^\ell$ for some $k, \ell \geq 0$, where $z$ is a subword of $ptuvq$. Then, by Corollary~\ref{within1}, since $ptuvq$ is jointly balanced 
with $\underline{x_{\alpha}}$, $\big| \#(z,1) - |z| \alpha \big| \leq 1$. Then $\#(s,1) = ik + i\ell + \#(z,1)$,
and therefore within $1$ of $ik + i\ell + |z|\alpha = (jk + j\ell + |z|) \alpha = |s| \alpha$. We have then shown that 
$a = p^{\infty} tuv q^{\infty}$ is jointly balanced with $\underline{x_{\alpha}}$. 
We note that since $a$ is jointly balanced with $\underline{x_{\alpha}}$, the biinfinite sequences $p^{\infty}$ and $q^{\infty}$ must be as well. 
We then claim that these sequences are in fact ($1$-)balanced. For any length $m$ which is not a multiple of $j$, 
Lemma~\ref{bal2choice} implies that every $m$-letter subword of $p^{\infty}$ or $q^{\infty}$ has either $\lfloor m\alpha \rfloor$ or 
$\lceil m\alpha \rceil$ $1$s. For any $m$ a multiple of $j$, since $p^{\infty}$ is periodic with period $j$, every $m$-letter subword of $p^{\infty}$
or $q^{\infty}$ has the same number of $1$s (namely $m\alpha$). Therefore, $p^{\infty}$ and $q^{\infty}$ are in fact balanced, and
by Lemma~\ref{balper}, must be in $O(\underline{x_{\alpha}})$ themselves. 

By Claim C1, define $x' \in X_{\alpha,2}$ with $a(x') = a$ and $b(x') = b(x)$. Clearly $(a(x'))([-N-n,N+n])$ is unchanged from $a(x)$ and so equals $u$,
and $(b(x'))([-N-n,N+n]) = v$. As argued before, these occurrences of $u,v$ force the first two coordinates of $x'([-N,N] \times [-n,n])$ to match those of $w'$, and the third coordinate on any rows of $x'$ with equal first and second coordinates can be changed to yield
$x''$ containing $w'$ (and thereby $w$). 
Since $a(x'')$ and $b(x'')$ both terminate with a shift of $\underline{x_{\alpha}}$, due to the periodicity of the first and
second coordinates of $x''$, there must be a subpattern of $x''$ of shape $[-n,n]^2$, call it $w''$, whose first and second
coordinates on every row are just subwords of $\underline{x_{\alpha}}$. 
Call the set of such patterns $S_n$. We have then shown that any pattern in 
$L_{[-n,n]^2}(X_{\alpha,2})$ coexists in a point of $X_{\alpha,2}$ with a pattern from the set $S_n$, 
and similarly one can prove that any pattern in $L_{[-n,n]^2}(X_{\alpha,3})$ coexists
in a point of $X_{\alpha,3}$ with a pattern from $S_n$. 

Finally, we claim that any two patterns $s,t \in S_n$ coexist in some point of $X_{\alpha,0} \subset X_{\alpha,2} \cup X_{\alpha,3}$.
We may without loss of generality assume that $n > j$. Define $s'$ and $t'$ to be the patterns given by the first two coordinates
of $s$ and $t$ respectively, and use Claim C1 to define $x \in X_{\alpha,0}$ with $a(x) = b(x) = \underline{x_{\alpha}}$. 
Then all possible ``phase shifts'' of the first and second coordinates appear in infinitely many rows, and so 
$s'$ and $t'$ appear infinitely many times in $x$; in particular, there are occurrences of them which share no row. 
Since $n > j$, in any rows where the first and second coordinates of 
those occurrences of $s'$ or $t'$ agree, the corresponding entire rows of $x$ have equal first and second coordinates.
Then, in any such rows, the third coordinate of $x$ can be changed (if necessary) to yield $x' \in X_{\alpha,0}$ containing
$s$ and $t$.

We have then proved (3b) from Theorem~\ref{blackbox} (with $N = 3$) for $X_{\alpha,2} \cup X_{\alpha,3}$; 
for any two patterns $w,w'$ in $L_{[-n,n]^2}(X_{\alpha,2} \cup X_{\alpha,3})$, each coexists with a pattern from $S_n$ in some point of
$X_{\alpha,2} \cup X_{\alpha,3}$, and then the two patterns from $S_n$ coexist in some point of $X_{\alpha,0} \subset X_{\alpha,2} \cup X_{\alpha,3}$. Finally, we apply Theorem~\ref{blackbox}
to see that $f(X_{\alpha,2} \cup X_{\alpha,3})$ has TCPE. 

\end{proof}

We are finally prepared to prove that $(f(X), \sigma_v)$ has ZTCPE but not TCPE.\\

\noindent
\textbf{Claim C8:} $(f(X), \sigma_v)$ does not have TCPE.

\begin{proof}

As in the corresponding proof from Theorem~\ref{1dsub}, we define a surjective factor map from $(f(X), \sigma_v)$ to the nontrivial zero entropy system $([0,1], \textrm{id})$. The map $\pi$ is defined as follows: for every $c \in f(X)$, $\pi(c)$ is defined to be the unique $\alpha$ so that 
$c \in f(X_\alpha)$. The arguments that $\pi$ is shift-invariant and surjective are the same as before. It remains only to show that $\pi$ is continuous, but this is simple; if $c_n \in f(X)$ approaches $c$, then $c_n$ is induced by $b_n$ approaching $b$, and in particular
$a(b_n)$ approaches $a(b)$. But then $a(b_n)$ and $a(b)$ are $2$-balanced sequences, and the proof that the slopes of $a(b_n)$ approach the slope of $a(b)$ is the same as in the one from Theorem~\ref{1dsub}. This implies that $\pi(c_n) \rightarrow \pi(c)$, and that $\pi$ is continuous, meaning that 
$(f(X), \sigma_v)$ does not have TCPE.

\end{proof}

\noindent
\textbf{Claim C9:} $(f(X), \sigma_v)$ has ZTCPE.

\begin{proof}
Again we proceed by showing that every factor map on $(f(X), \sigma_v)$ factors through $\pi$. 
Consider any surjective factor map $\psi: (f(X), \sigma) \rightarrow (Y, S_v)$ where $h(Y,S_v) = 0$ and $Y$ is a zero-dimensional topological space. We must show that $|Y| = 1$. 

For every $\alpha$, $\psi(f(X_{\alpha})) \subset Y$, and so clearly $h(\psi(f(X_{\alpha})), S_v) = 0$. Therefore, by Claim C3 above, for $\alpha \notin \mathbb{Q}$, $\psi(f(X_{\alpha}))$ is a single point, which we denote by $g(\alpha)$.

Similarly, for any $\alpha \in \mathbb{Q}$, by Claim C4, $\psi(f(X_{\alpha,0}))$ consists of a single point, which we denote by $g(\alpha)$.
For $\alpha \in \mathbb{Q} \setminus \{0,1\}$, Claim C7 implies that $\psi(f(X_{\alpha,2}) \cup f(X_{\alpha,3}))$ consists of a single point. 
Since $f(X_{\alpha,0}) \subset f(X_{\alpha,2}) \cup f(X_{\alpha,3})$, this point must also be $g(\alpha)$. 

We have now defined $g$ on all of $[0,1]$, and claim that it is continuous. This is done similarly as in the corresponding proof from 
Theorem~\ref{1dsub}, but the third coordinate causes some technical difficulties. Consider any sequence $\alpha_n$ which approaches a limit 
$\alpha$ from above. Then, by Claim C1, define $x_n \in X_{\alpha_n}$ by taking $a(x_n)$ and $b(x_n)$ to both be the lower characteristic sequence 
$\underline{x_{\alpha_n}}$; note that if $\alpha_n \in \mathbb{Q}$, then in addition $x_n \in X_{\alpha_n,0}$. 
For any row where the third coordinate is not forced by the first two (including the $0$th row), label the third coordinate by all $1$s if it is a nonnegatively indexed row, and by all $0$s if it is a negatively indexed row.
From Lemma~\ref{charconv}, the first two coordinates of $x_n$ clearly approach a limit, and any point 
$x$ with those first two coordinates would have $a(x) = b(x) = \underline{x_{\alpha}}$. It remains to show that the third coordinates
of $x_n$ actually converge. To see this, choose any row, say the $k$th, and let's examine what happens to the third coordinates of $x_n$ along that row as $n$ increases. 
If $\sigma^k \underline{x_{\alpha}} \neq \underline{x_{\alpha}}$, then there is a place in the $k$th row where the first and second coordinates are unequal for large enough $n$, meaning that the third coordinate is forced by the first two for large $n$ and therefore must approach a limit since the first two do. If $\sigma^k \underline{x_{\alpha}} = \underline{x_{\alpha}}$, then $k \alpha \in \mathbb{Z}$ and $\alpha \in \mathbb{Q}$. 
Since $\alpha_n > \alpha$, if we denote by $i_n,j_n$ the negative and positive indices at which $\underline{x_{\alpha_n}}$ 
and $\underline{x_{\alpha}}$ first differ, then $\underline{x_{\alpha_n}}(i_n) = \underline{x_{\alpha_n}}(j_n) = 1$ and 
$\underline{x_{\alpha}}(i_n) = \underline{x_{\alpha}}(j_n) = 0$.
Choose $n$ large enough that $|i_n|, |j_n| > k$. If $k > 0$, then the first and second coordinates of the $k$th row of
$x_n$ agree from $i_n + k + 1$ to $j_n - 1$, and at $j_n$ the first coordinate has a $1$ and the second has a $0$. This forces 
the third coordinate at $j_n - 1$ to be a $1$, and since the first and second coordinates agree from $i_n + k + 1$ to $j_n - 1$,
the third coordinate is $1$ throughout that range. As $n \rightarrow \infty$, $|i_n|, |j_n| \rightarrow \infty$, and
so the third coordinate on the $k$th row approaches all $1$s. Similarly, for $k < 0$, the third coordinate will approach all $0$s.
Therefore, $x_n$ does in fact approach a limit $x$, where all non-negatively indexed rows with nonforced third coordinate 
have that coordinate labeled with all $1$s, and all similar negatively indexed rows have third coordinate all $0$s. This 
point $x$ is in $X_{\alpha}$ by definition, and in $X_{\alpha,0}$ if $\alpha \in \mathbb{Q}$. We may create 
$y_n \in f(X_{\alpha_n})$ induced by $x_n$ and $y \in f(X_{\alpha})$ (or $f(X_{\alpha,0})$) induced by $x$ for which
$y_n \rightarrow y$; just use the same ``ribbon structure'' for all of the points. Then, by continuity of $\psi$,
$\psi(y_n) \rightarrow \psi(y)$. However, $\psi(y_n) = g(\alpha_n)$ and $\psi(y) = g(\alpha)$, and so we've shown that
$g(\alpha_n) \rightarrow g(\alpha)$, and therefore that $g$ is continuous from the right.
A similar argument using upper characteristic sequences and third coordinate $0$ in the upper half-plane and $1$ in the lower-half plane
shows that $g$ is continuous from the left, and therefore continuous.

The only points of $f(X)$ which have not yet been considered are those in $f(X_{\alpha,1}) \setminus f(X_{\alpha,0})$.
(It may look as if we've ignored $f(X_{\alpha,2}) \cup f(X_{\alpha,3})$ for $\alpha \in \{0,1\}$, but by Claim C6 such points
are already contained in $X_{\alpha,1}$.) By Claim C5 above, every point $y \in f(X_{\alpha,1}) \setminus f(X_{\alpha,0})$ 
can be written as a limit from points of $f(X_{\alpha_n})$ for some sequence of irrationals
$\alpha_n \rightarrow \alpha$. But then $\psi(y)$ is the limit of $g(\alpha_n)$, and by continuity of $g$, this
implies that $\psi(y) = g(\alpha)$. We have then shown that for every $\alpha$, $\psi(f(X_{\alpha})) = g(\alpha)$.

Since $g$ is continuous on $[0,1]$, $g([0,1]) = Y$ must be connected (as the continuous image of a connected set), and the only connected subsets of $Y$ are singletons. We have therefore shown that $g$ is constant, and so 
$|Y| = 1$. Since $\psi$ was arbitrary, $(X, \sigma_v)$ has ZTCPE.

\end{proof}

We've shown that the $\mathbb{Z}^2$-SFT $(f(X), \sigma_v)$ has ZTCPE but not TCPE, completing the proof of Theorem~\ref{2dSFT}.

\end{proof}

We end by briefly remarking on a comment made in the introduction; by Theorems~\ref{ZTCPEthm} and \ref{TCPEthm},
for $X$ as in Theorem~\ref{2dSFT}, any two patterns in $L_{[-n,n]^2}(f(X))$ must be chain exchangeable, but the maximum
number of required exchanges between two such patterns must increase as $n \rightarrow \infty$. 
This can be seen informally without reference to TCPE or ZTCPE as follows.
As shown in the proof of Theorem~\ref{blackbox},
two patterns $v,w \in L_{[-n,n]^2}(f(X))$ induced by $v',w' \in L(X)$ are exchangeable only if
$v'$ and $w'$ appear in the same point of $X$. Such $v'$ and $w'$ are essentially determined
by pairs of jointly balanced words. A balanced word of length $n$ generally determines the slope
of a balanced sequence containing it within a tolerance which approaches $0$ as $n \rightarrow \infty$.
So, two pairs of jointly balanced words of length $n$ may appear in the same pair 
of jointly balanced sequences only if their frequencies of $1$s are close enough. 
Therefore, if $v',w'$ have frequencies of $1$s quite far apart, then the number of exchanges required to get
from $v$ to $w$ will increase with $n$.

\bibliographystyle{plain}
\bibliography{ZTCPE}

\end{document}